\documentclass[a4paper, 10pt]{article}

\usepackage[T1]{fontenc}
\usepackage[utf8]{inputenc}

\usepackage{amssymb}
\usepackage{amsmath}
\usepackage{amsfonts}
\usepackage{amsthm}
\usepackage{mathtools}
 
\usepackage{comment}
\usepackage{url}
\usepackage{float}
\usepackage[dvipsnames]{xcolor}
\usepackage{enumerate}
\usepackage{tikz}
\usepackage{caption}
\usepackage{cases}

\usepackage{changepage}
\usepackage{adjustbox} 

\usepackage{booktabs}
\usepackage{multirow}

\newcommand{\E}{\mathbb{E}}

\DeclareMathOperator{\RER}{RER}
\DeclareMathOperator{\Null}{Null}

\theoremstyle{plain}%
\newtheorem{theorem}{Theorem}[section]
\newtheorem{lemma}[theorem]{Lemma}
\newtheorem{proposition}[theorem]{Proposition}

\theoremstyle{definition}
\newtheorem{definition}[theorem]{Definition}

\theoremstyle{remark}
\newtheorem{remark}[theorem]{Remark}

\usepackage{url}
\usepackage{hyperref}% 

\usepackage{dsfont}
\usepackage{bbold}

\title{Color representations of Ising models}

\author{Malin P. Forsstr\"om\footnote{malinpf@kth.se}\medskip \\ 
\emph{Department of Mathematics, KTH Royal Institute of Technology,} \\
\emph{100 44 Stockholm, Sweden.}}
\date{\today}

\begin{document}

\maketitle

\begin{abstract}
	In~\cite{st2017}, the authors introduced the so called \emph{general divide and color models}. One of the most well-known examples of such a model is the Ising model with external field \( h = 0 \), which has a color representation given by the random cluster model. In this paper, we give necessary and sufficient conditions for this color representation to be unique. We also show that if one considers the Ising model on a complete graph, then for many \( h > 0 \), there is no color representation. This shows in particular that any generalization of the random cluster model which provides color representations of Ising models with external fields cannot, in general, be a generalized divide and color models. Furthermore, we show that there can be at most finitely many \( \beta > 0 \)  at which the random cluster model can be continuously extended to a color representation for \( h \not = 0 \).
\end{abstract}

\section{Introduction}

A simple mechanism for constructing random variables with a positive dependency structure is the so called \emph{generalized divide and color model}. This model was first introduced in~\cite{st2017}, but similar constructions had already arisen in many different contexts.

\begin{definition}
Let  \( S  \) be a finite set. A \(\{0,1\}^S\)-valued random variable \( X \coloneqq (X_i)_{i \in S} \) is called a \emph{generalized divide and color model} if \( X \) can be generated as follows. 
\begin{enumerate}
    \item Choose a random partition \( \pi \) of $S$ according to some arbitrary {distribution $\mu$}.
    \item Let \( \pi_1, \ldots, \pi_m \) be the partition elements of \( \pi \). Independently for each \( i \in \{ 1,2, \ldots, m\}\), pick a "color" \( c_i \sim (1-p)\delta_0 + p \delta_1 \), and assign all the elements in \(  \pi_i\) the color \( c_i \) by letting \( X_j = c_i \) for all \( j \in \pi_i \).
%Given the partition \( \pi \), assign a "color" in \( \{0,1 \}\) to each element in \( S \) as follows. With probability \( p \), assign {\it all} the variables in a partition element the value \( 1 \) and with probability \( 1-p \) assign {\it all} the variables the value \( 0 \). Do this independently for different partition elements.
     
    %Independently of the chosen partition \( \pi \),  and independently for different partition elements in \( \pi \), assign, with probability \( p \), {\it all} the variables in a partition element the value \( 1 \) and with probability \( 1-p \) assign {\it all} the variables the value \( 0 \).
\end{enumerate}{} 
The final \(\{0,1\}\)-valued 
process \( X \) is called the \emph{generalized divide and color} model associated to $\mu$ and $p$, and we  say that $\mu$ is a \emph{color representation} of $X$.
\end{definition}

As detailed in \cite{st2017}, many processes in probability theory
are generalized divide and color models; one of the most prominent examples being the Ising model with no external field. To define this model, let \( G = (V,E ) \) be a finite connected graph with vertex set \( V \) and edge set \( E \). We say that a random vector \( X = (\sigma_i)_{i \in V} \in \{ 0,1 \}^{V} \) is an \emph{Ising model} on \( G \) with interaction parameter \( \beta > 0 \) and external field \( h \in \mathbb{R} \) if \( X \) has probability density function \( \nu_{G, \beta,h }   \) proportional to 
\[
\exp \Bigl( \beta \sum_{\{ i,j\} \in E} \bigl(\mathbb{1}_{\sigma_i = \sigma_j}-\mathbb{1}_{\sigma_i \neq \sigma_j}\bigr) + h \sum_{i \in V} \bigl(\mathbb{1}_{\sigma_i=1} - \mathbb{1}_{\sigma_i = 0} \bigr) \Bigr).
\] 
The parameter \( \beta \) will be referred to as the \emph{interaction parameter} and the parameter \( h \) as the strength of the \emph{external field}.
We will write \( X^{G,\beta, h}  \) to denote a random variable with  \( X^{G,\beta,h} \sim \nu_{G,\beta,h} \).
It is well known that the Ising model has a color representation  when \( h = 0 \) given by the  random cluster model. To define the \emph{random cluster model} associated to the Ising model we first, for \( G = (V,E) \) and  \( w \in \{ 0,1 \}^E \), define  \( E_w \coloneqq \{ e \in E \colon w_e = 1 \}\) and note that this defines a partition \( \pi[w] \) of \( V \), where \( v ,v' \in V \) are in the same partition element of \( \pi[w] \) if and only if they are in the same connected component of the graph \( (V,E_w) \). Let \( \|\pi[w]\| \) be the number of partition elements of \( \pi[w] \) and let \( \mathcal{B}_V \) denote the set of partitions of \( V \). For \( r \in (0,1) \) and \( q \geq 0 \), the random cluster model \(   \mu_{G,r,q} \) is defined by
\begin{equation*}
	\mu_{G,r,q}(\pi')=\frac{1}{Z'_{G,r,q}} \sum_{w \in \{ 0,1\}^E \colon\atop \pi[w]=\pi'} \biggl[ \, \prod_{e \in E} r^{w(e)}(1-r)^{1-w(e)} \biggr] q^{\|\pi[w]\|}, \quad \pi' \in \mathcal{B}_V.
\end{equation*}
where   \( Z'_{G,r,q}  \) is a normalizing constant ensuring that this is a probability measure. It is well known  (see e.g.\ \cite{g2006}) that  if one sets \( r = 1 - \exp (-2 \beta) \), \( q = 2 \) and \( p = 1/2 \), then \( \mu_{G,r,q} \) is a color representation of \( X^{\beta,0}  \). To simplify notation, we will write \( \mu_{G,r} \coloneqq \mu_{G,r,2} \).

Since many properties of the Ising model with \( h=0 \) has been understood by using a color representation (given by the random cluster model, see e.g. \cite{accn1988,bc1996,g2006}), it is natural to ask if there is a color representation also when \( h > 0 \). Moreover, Theorems~1.2 and~1.4 in~\cite{fs2019}, which state that a random coloring of a set can have more than one color representation, motivates asking whether there are any color representation when \( h=0 \) which is different  from the random cluster model. 
The main objective of this paper is to provide partial answers to these questions by investigating how generalized divide and color models relate to some Ising models, both in the presence and absence of an external field.

In order to be able to present these results, we will need some additional notation. Let \( S \) be a finite set. For any measurable space \( (S , \sigma(S))\), we let  \( \mathcal{P}(S) \) denote the set of probability measures on \( (S, \sigma(S))\). When \( S = \{ 0,1 \}^T \) for some finite set \( T \), then we always consider the discrete \( \sigma \)-algebra, i.e.\ we let  \( \sigma(\{0,1 \}^T) \) be the set of all subsets of \( \{ 0,1\}^{T}\).
Recall the we let \( \mathcal{B}_S \) denote the set of partitions of \( S \). If \( \pi \in \mathcal{B}_S \) and \( T \subseteq S\), we let \( \pi|_T \) denote the partition of \( T \)  induced from \( \pi \) in the natural way. On \( \mathcal{B}_S \) we consider the \( \sigma \)-algebra \( \sigma(\mathcal{B}_S) \) generated by \( \{ \pi|_T \}_{T \subseteq S,\, \pi \in \mathcal{B}_S}\). In analogy with~\cite{st2017}, we let \( \RER_S\) denote the set of all probability measures on \( (\mathcal{B}_S, \sigma(\mathcal{B_S}))\) (RER stands for random equivalence relation). For a graph \( G\) with vertex set \( V \) and edge set \( E \), we let \( \RER_V^G \) denote the set of probability measures \( \mu \in \RER_V \) which has support only on partitions \( \pi \in \mathcal{B}_V\) whose partition elements induce connected subgraphs of \( G \).
For  each \( p \in (0,1) \), we now introduce the mapping \( \Phi_p \) from \( \RER_S \) to the set of probability measures on \( \{ 0,1 \}^S \) as follows. Let \( \mu \in \RER_S \). Pick \( \pi \) according to \( \mu \). Let \( \pi_1, \pi_2, \ldots, \pi_m \) be the partition elements of \( \pi \). Independently for each \( i \in \{ 1,2, \ldots, m \}\), pick \( c_i \sim (1-p) \delta_0 + p\delta_1 \) and let \( X_j = c_i \) for all \( j \in \pi_i \). This yields a random vector \( X = (X_i)_{i \in S} \) whose distribution will be denoted by \( \Phi_p(\mu) \). The random vector \( X \) will be referred to as a generalized divide and color model, and the measure \( \mu \) will be referred to as a color representation of \( X \) or \( \Phi_p(\mu) \). 
Note that \( \Phi_p \colon \mathcal{B}_S \to \mathcal{P}(\{ 0,1\}^S) \).
We will say that a probability measure   \(\nu \in  \mathcal{P}(\{ 0,1\}^S) \) has a color representation if there is a measure \( \mu \in \RER_S \) and \( p \in (0,1) \) such that \( \nu = \Phi_p(\mu) \). For \( \nu \in \mathcal{P}(S) \), we let \( \Phi^{-1}(\nu) \coloneqq \{ \mu \in RER_S \colon \Phi_p(\mu) = \nu \}\). Then \( \nu \) has a color representation if and only iff there is \( p \in (0,1) \) such that \( \Phi_p^{-1}(\nu) \) is non-empty.
By Theorems~1.2~and~1.4 in~\cite{fs2019}, \( \Phi_p^{-1}(\nu) \) can be non-empty if and only if the one-dimensional marginals of \( \nu \) are all equal to \( (1-p) \delta_0 + p \delta_1 \). From this it immediately follows that for any graph \( G \) and any \( \beta > 0 \), we have \( \Phi_p^{-1}(\nu_{G,\beta,0}) = \emptyset \) whenever \( p \neq 1-e^{-2\beta} \).

 Our first result is the following theorem, which states that for any finite graph \( G \) and any \( \beta > 0 \), \( X^{G,\beta,0} \) has at least two distinct color representations.

\begin{theorem}\label{theorem: color representations with no external field}
%Let \( X^{G, \beta,0}\) be the Ising model   on a connected graph \( G \)  with at least three vertices. Then 
%\( X^{\beta,0} \) has at least two distinct color representations.
Let \( n \in \mathbb{N} \) and let  \( G \) be a connected graph with \( n \geq 3\) vertices. Further, let \( \beta > 0 \). Then there are at least two distinct probability measures \( \mu,\mu' \in RER_{V(G)} \) such that \( \Phi_{1/2}(\mu) = \Phi_{1/2}(\mu') = \nu_{G,\beta,0}\).
Furthermore, if \( G \) is not a tree, then there are at least two distinct probability measures \( \mu,\mu' \in RER_{V(G)}^G \) such that \( \Phi_{1/2}(\mu) = \Phi_{1/2}(\mu') = \nu_{G,\beta,0}\).

%then there are at least two distinct color representations of \( X^{\beta,0} \) which has support only at partitions with partition elements inducing connected subgraphs of \( G \).
\end{theorem}

We remark that if a graph \( G \) has only one or two vertices and \( \beta > 0 \), then it is known from Theorem 2.1 in~\cite{st2017} (see also Theorems 1.2 and 1.4 in~\cite{fs2019}) that \( \Phi_{1/2}^{-1}(\nu_{G, \beta ,0}) = \{ \mu_{G,1-\exp(-2\beta)} \} \). In other words, when \( h = 0 \), the Ising model   \(   X^{G, \beta,h}\) has a unique color representation, given by the random cluster model \( \mu_{G,1-\exp(-2\beta)} \).
To get an intuition for what should happen when \( h > 0 \), we first look at a few toy examples. One of the simplest such examples is the Ising model on a complete graph with three vertices. The following result was also included as Remark~7.8(iii) in~\cite{st2017} and in~\cite{fs2019} as Corollary~1.8.
\begin{proposition}\label{theorem: Ising on K3}
Let \( G \) be the complete graph on three vertices. Let \( \beta > 0 \) be fixed. For each  \( h > 0 \), let \( p_h \in (0,1 ) \) be such that the marginal distributions of \( X^{G, \beta, h} \) are given by \( (1-p_h)\delta_0 + p_h\delta_1 \). Then the following holds.
\begin{enumerate}[(i)]
\item  For each \( h > 0 \), we have \( \bigl| \Phi_{p_h}^{-1}(\nu_{G, \beta,h}) \bigr| = 1\), i.e.\ \( X^{G, \beta , h}  \) has a unique color representation for any \(h > 0 \).
\item  For each \( h>0 \), let \( \mu_h \) be defined by \( \{ \mu_h \} =  \Phi_{p(\beta,h)}^{-1}(\nu_{G, \beta,h}) \). Then \( \mu_0(\pi) \coloneqq \lim_{h \to 0} \mu_h(\pi) \) exists for all \( \pi \in \mathcal{B}_{V(G)} \), and \( \Phi_{1/2}(\mu_0) = \nu_{G,\beta,0} \). However,  \(  \mu_0 \neq \mu_{G,1-e^{-2\beta}}\). 
\end{enumerate}
\end{proposition}

Interestingly, if we increase the number of vertices in the underlying graph by one, the picture immediately becomes more complicated.

\begin{proposition}\label{theorem: Ising on K4}
Let \( G \) be the complete graph on four vertices. Let \( \beta > 0 \) be fixed. For each  \( h > 0 \), let \( p_h \in (0,1 ) \) be such that the marginal distributions of \( X^{G, \beta, h} \) are given by \( (1-p_h)\delta_0 + p_h\delta_1 \). 
To simplify notation, set  \( x  \coloneqq e^{2\beta} \) and \( y_h  \coloneqq e^{2h} \). Then \( X^{G, \beta, h} \) has a color representation if and only if 
\begin{equation}\label{eq: K4}
\begin{split}
x^{5} &+ 3 x^2 y_h + 4 x y_h^2 - 2 x^3 y_h^2 + x^{5} y_h^2 - 3 y_h^3 + 
  7 x^2 y_h^3 - x^4 y_h^3 - x^{6} y_h^3 
  \\ &+ 4 x y_h^4 - 2 x^3 y_h^4 + x^{5} y_h^4 + 
  3 x^2 y_h^5 + x^{5} y_h^6 \geq 0.
  \end{split}
\end{equation}
%In particular, if we let \( x_0 \) be the largest positive root to \( -1+4x^2-x^4 \) (i.e. \( x_0 = \sqrt{2 + \sqrt{3}} \)).
In particular, if we let \( \beta_0 \coloneqq \log(2 + \sqrt{3})/2 \), then the following holds.
\begin{enumerate}[(i)]
\item If \(\beta< \beta_0\), then \( X^{G, \beta,h} \) has a color representation for all sufficiently small \( h>0 \),
\item If \(\beta> \beta_0 \), then \( X^{G, \beta,h} \) has no color representation for any sufficiently small \( h>0 \).
\end{enumerate}
Moreover, there is no decreasing sequence \( h_1, h_2, \ldots \) with \( \lim_{n \to \infty} h_n = 0 \) and \( \mu_{h_n} \in \Phi_{p_{h_n}}^{-1}(\nu_{G, \beta,h}) \) such that \( \lim_{n \to \infty} \mu_{h_n}(\pi) = \mu_{G, 1-e^{-2\beta}}(\pi) \) for all \( \pi \in \mathcal{B}_{V(G)} \). In other words, the random cluster model does not arise as a subsequential limit of color representations of \( X^{G, \beta,h} \) as \( h \to 0 \), for any \( \beta > 0 \).
\end{proposition}
Interestingly, this already shows that there are graphs \( G \) and parameters \( \beta,h > 0 \) such that the corresponding Ising model \( X^{G, \beta,h} \)  does not have any color representations.

The proof of Proposition~\ref{theorem: Ising on K4} can in principle be extended directly to complete graphs with more than four vertices, but it quickly becomes computationally  heavy and the analogues of~\eqref{eq: K4} become quite involved (see Remark~\ref{remark: K5} for the analogue expression for a complete graph on five vertices). In Figure~\ref{fig: K45}, we draw the set of all pairs \( (\beta,h )\in \mathbb{R}_+^2\) which satisfies the inequality in~\eqref{eq: K4} together with the corresponding set for a complete graph on five vertices.

\begin{figure}[ht]
\centering
\includegraphics[width=0.6\linewidth]{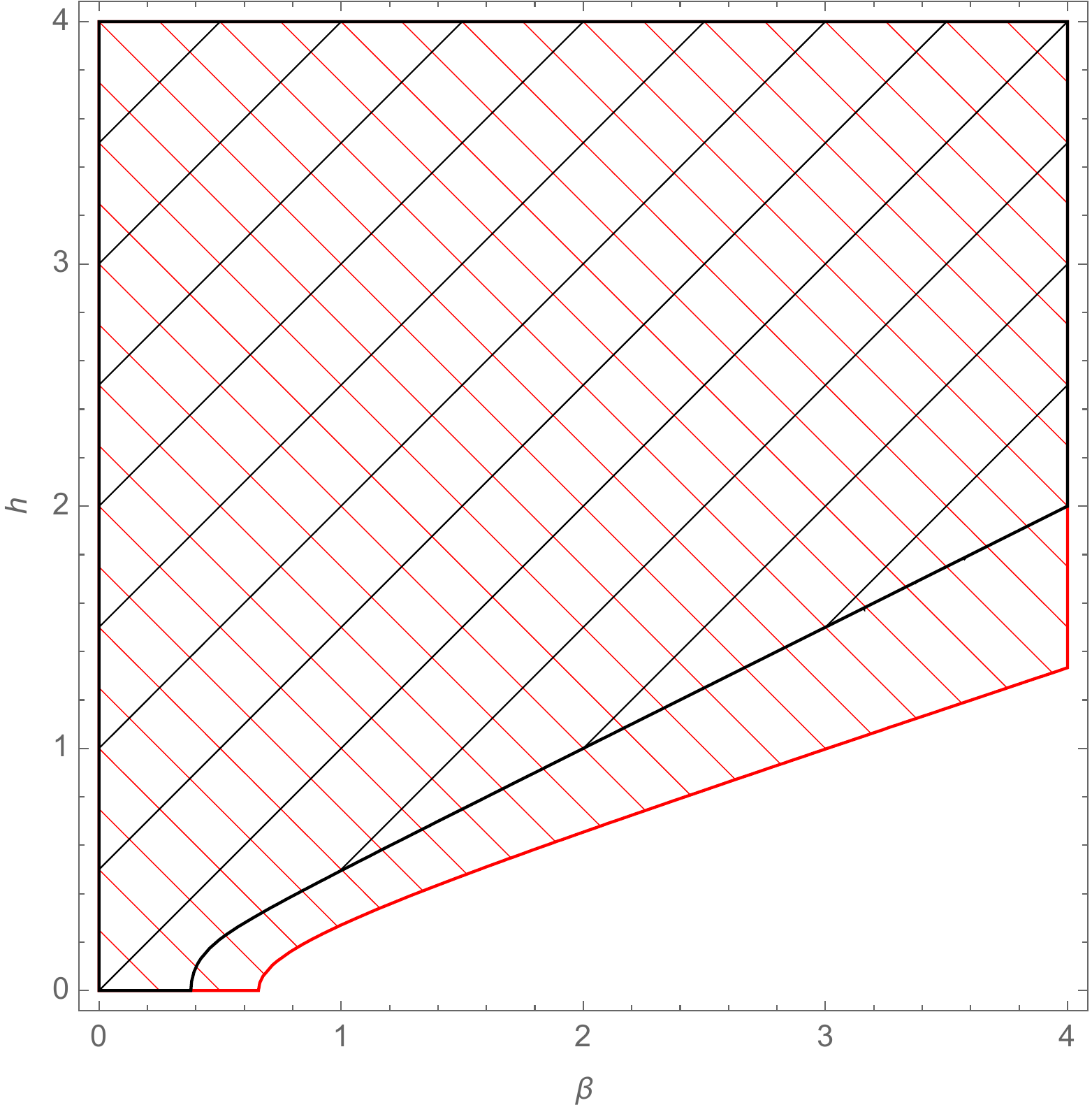}
\caption{The figure above shows the sets of all \( (\beta, h) \in \mathbb{R}_+^2 \) which are such that the Ising model \( X^{G, \beta,h}\), for \( G \) being a complete graph on four vertices (red) and five vertices (black) respectively, has at least one color representation (see Proposition~\ref{theorem: Ising on K4} and Remark~\ref{remark: K5}). }\label{fig: K45}
\end{figure}
Figure~\ref{fig: K45}, together with the previous two propositions, suggests that the following conjectures should  hold for all complete graphs \( G \) on at least four vertices.\begin{enumerate}[I.]
	\item If \( \beta>0 \) is sufficiently small, then \( X^{G, \beta,h } \) has a color representation for all \( h \in \mathbb{R}\).
	\item For each \( \beta > 0 \), \( X^{G, \beta, h } \) has a color representations for all sufficiently large \(h \in \mathbb{R}\).
	\item If \( \beta \) is sufficiently large, then \( X^{G, \beta,h }\) has no color representation for any sufficiently small \( h >0 \).
	\item If \( \beta,h > 0 \) and \( X^{G, \beta, h} \) has a color representation, then so does \( X^{G, \beta',h} \) and \( X^{G, \beta,h' } \) for all \( h' > h \) and \( \beta' \in (0, \beta) \).
	\item The random cluster model corresponding to \( X^{G, \beta,0}\)  does not arise as a subsequential limit of color representations of \( X^{G, \beta,h} \) as \( h \to 0 \).
\end{enumerate}

Our next result concerns the last of these conjectures.

\begin{theorem}\label{theorem: rcm limit}
Let \( n \in  \mathbb{N} \) and let \( G \) be a connected and vertex-transitive graph with \( n \) vertices. For each \( \beta \geq 0 \) and \( h > 0 \),   let \( p_{\beta,h} \in (0,1 ) \) be such that the marginal distributions of \( X^{G, \beta, h} \) are given by \( (1-p_{\beta,h})\delta_0 + p_{\beta,h}\delta_1 \).  Then there is a set \( B \subseteq \mathbb{R}_+ \) with \( |B| \leq n(n-1)\) such that for all \( \beta \in \mathbb{R}_+ \backslash B \), all sequences \( h_1 > h_2 > \ldots \) with \( \lim_{m \to \infty} h_m = 0 \) and all sequences \( \mu_m \in \Phi_{p_{\beta,h_m}}^{-1}(\nu_{G,\beta,h}) \), there is a partition \( \pi \in \mathcal{B}_{V} \) such that \( \lim_{m \to \infty} \mu_m(\pi) \not \neq \mu_{G, 1-e^{-2\beta}} (\pi)\).
In other words, the random cluster model on \( G \)  can  arise as a subsequential limit when \( h \to 0 \) of color representations of \( X^{G, \beta,h} \)  for at most \( n(n-1) \) different values of \( \beta \).
\end{theorem}

 Interestingly, Theorem~\ref{theorem: rcm limit} does not require \( n \) to be large. The set of exceptional values for \( \beta \) where the random cluster model could arise as a limit is a consequence of the proof strategy used, and could possibly be shown to be empty by using a different proof.

Our next   result shows that the third of the conjectures above is true when \( n \) is sufficiently large.
\begin{theorem}\label{theorem: no cr when critical or supercritical}
Let \( n \in  \mathbb{N} \) and let \( G \) be the complete graph on \( n \) vertices. 
Further, let  \( \hat \beta  > 0\) and  \( \beta \coloneqq \hat \beta / n \). 
 If \( \hat \beta \geq \hat \beta_c =  1 \) and \( n \) is sufficiently large, then \( X^{G, \beta, h} \) has no color representation for any sufficiently small \( h >0\).
\end{theorem}
As a consequence of this theorem, the fifth conjecture above is true when \( \beta > 1/n \) and \( n \) is sufficiently large.

Our last result gives a partial answer to the second conjecture.  
\begin{theorem}\label{theorem: very large h}
Let \( n \in  \mathbb{N} \) and let \( G \) be the complete graph on \( n \) vertices. Let \( h > 0 \) and  \( \beta = \beta(h) \) be such that \( (n-1) \beta(h) < h \).
	Then  \( X^{G, \beta(h),h} \) has a color representation for all sufficiently large \(h \).
\end{theorem}

Simulations suggests that the previous result should be possible to extend to \( \beta(h) < h \). This is a much stronger statement, especially for large \(n \), and would  require a different proof strategy. The assumption that \( (n-1)\beta (h) \leq h \), made in Theorem~\ref{theorem: very large h}, is however a quite natural condition, since this exactly corresponds to that \( \nu_{G, \beta,h}(1^S0^{[n]\backslash S}) \) is decreasing in \( |S| \), for \( S \subseteq [n] \) (here \( 1^S 0^{[n]\backslash S}\) denotes the binary string \( x \in \{ 0,1 \}^n\) with  \( x(i) = 1 \) when \( i \in S \) and and \( x(i) = 0 \) when \( i \not \in S \)).

The rest of this paper will be structured as follows. In Section~\ref{section: background}, we give the background and definitions needed for the rest of the paper. In Section~\ref{sec: no external field}, we give a proof of Theorem~\ref{theorem: color representations with no external field}. In Section~\ref{section: small complete graphs}, we give proofs of Propositions~\ref{theorem: Ising on K3} and~\ref{theorem: Ising on K4}, and also discuss what happens when \( G \) is the complete graph on five vertices. Next, in Section~\ref{section: small h}, we prove Theorem~\ref{theorem: rcm limit} and Theorem~\ref{theorem: no cr when critical or supercritical} and in Section~\ref{section: large h}, we give a proof of Theorem~\ref{theorem: very large h}. Finally, in Section~\ref{section: technical lemmas}, we state and prove a few technical lemmas which are used throughout the paper.

\section{Background and notation}\label{section: background}

The main purpose of this section is to give definitions of the notation used throughout the paper, as well as some more background to the questions studied.

\subsection{The original divide and color model}

When the generalized divide and color model was introduced in~\cite{st2017}, it was introduced as a generalization of the so-called {divide and color model}, first defined by H\"aggstr\"om in~\cite{oh1999}. To define this family of models, let \( G  \) be a finite graph with vertex set \( V \) and edge set \( E \). Let \( r\in (0,1) \), and let \( \lambda \) be a finitely supported probability measure on \( \mathbb{R} \). 
% Let \( \mu \in RER_{V}\) be the measure on partitions corresponding to the connected components of Bernoulli bond percolation with parameter \( r \) on \( G \). 
The \emph{divide and color model} \( X = (X_v)_{v \in V}\) (associated to \( r \) and \( \lambda \)) is the random coloring of \( V \) obtained as follows.
\begin{enumerate}
	\item Pick \( \pi \sim \mu_{G,r,1} \). 
	\item Let \( \pi_1, \ldots, \pi_m \) be the partition elements of \( \pi \). Independently for each \( i \in \{ 1,2, \ldots, m\}\), pick \( c_i \sim \lambda \), and assign all the vertices \( v \in \pi_i\) the color \( c_i \) by letting \( X_v = c_i \) for all \( v \in \pi_i \)
\end{enumerate}
Note that if \( \lambda = (1-p)\delta_0 + p \delta_1 \), then \( X \sim \Phi_p(\mu_{G,r,1}) \).

Since its introduction in~\cite{oh1999}, properties of the divide and color model have been studied in several papers, including e.g.~\cite{bbt2013, bcm2009, g}. Several closely related models, which all in some way generalize the divide and color model by considering more general measures \( \lambda \) have also been considered (see e.g.~\cite{ab2010}).

We stress that this is not the model we discuss in this paper. To avoid confusion between the divide and color model and the generalized divide and color model, we will usually talk about color representations rather than generalized divide and color models.

\subsection{Generalizations of the coupling between the Ising model and the random cluster model}
When \( h >0  \), there is a generalization of the random cluster model (see e.g.\ \cite{bbck2000}) from which the Ising model can be obtained by independently  assigning colors to different partition elements. This model has been shown to have properties which can be used in similar ways as analogue properties of the random cluster model.
However, since this model use different color probabilities for different partition elements, it is not a generalized divide and color models.
On the other hand, Proposition~\ref{theorem: Ising on K4} shows that there are graphs \( G \) and parameters \( \beta,h > 0 \) such that \( X^{G,\beta,h}\) has no color representation. This  motivates considering  less restrictive generalizations of the random cluster model, such as the one given in~\cite{bbck2000}.

\subsection{General notation}

Let \( \mathbb{1}\) denote the indicator function, and for each \( n \in  \mathbb{N} \), define \( [n] \coloneqq \{ 1,2, \ldots, n \} \).

 When \( G \) is a graph, we will let \( V(G) \) denote the set of vertices of \( G \) and \( E(G) \) denote its set of edges. For each graph \( G \) with \( |V(G)| = n\), we assume that a bijection from   \( V(G) \) to  \( [n] \) is fixed, and in this way identify binary strings \( \sigma \in \{ 0,1 \}^{V(G)} \) with the corresponding binary strings in \( \{ 0,1 \}^n \).  The complete graph on \( n \) vertices will be denoted by \( K_n \).

For all finite sets \( S \),  disjoint sets \( T,T' \subseteq S \) and \( \sigma = (\sigma_i)_{i \in S}\in \{ 0,1 \}^S\), we now make the following definitions. Let \( \sigma|_T \) denote the restriction of \( \sigma \) to \( T \).
Write \( \sigma|_T \equiv 1 \) if \( \sigma_i = 1 \) for all \( i \in T \), and analogously write \( \sigma|_T \equiv 0 \) if \( \sigma_i = 0 \) for all \( i \in T \).
We let \( 1^{T}0^{T'} \) denote the unique binary string  \( \sigma \in  \{ 0,1 \}^{T \cup T'} \) with   \( \sigma|_T \equiv 1 \)  and   \( \sigma|_{T'} \equiv 0 \). Whenever \( \nu \) is a signed measure  on \( (\{ 0,1 \}^S,\sigma(\{0,1 \}^S) \), we write \( \nu(1^T) \coloneqq \sum_{T' \subseteq [n]\backslash T} \nu(1^{T \cup T'}0^{S \backslash (T \cup T')}) \).  
We let  \( \| \sigma \|  \coloneqq  \sum_{i \in S} \sigma_i \) and define  \( \chi_T(\sigma) \coloneqq \prod_{i \in T} (-1)^{\mathbb{1}_{\sigma_i = 0}}= (-1)^{n-\| \sigma \|} \). %

We now give some notation for working with set partitions. To this end, recall that when \( S \) is a finite set we let \( \mathcal{B}_S \) denote the set of partitions of \( S \). If \( S = [n] \) for some \( n \in \mathbb{N} \), we let \( \mathcal{B}_n \coloneqq \mathcal{B}_{[n]}\). 
If \( \pi \in \mathcal{B}_S \) has partition elements \( \pi_1 \), \( \pi_2 \), \ldots, \( \pi_m \), we write \( \pi = (\pi_1, \ldots, \pi_m) \). 
Now assume that a finite set \( S \), a partition \( \pi \in \mathcal{B}_S \) and a binary string  \( \sigma \in \{ 0,1 \}^S\) are given.
We write \( \pi \lhd \sigma \) if \( \sigma \) is constant on the partition elements of \( \pi \). If \( \pi \lhd \sigma \), \( \pi_i \) is a partition element of \( \pi \) and \( j \in \pi_i \), we write \( \sigma_{\pi_i} \coloneqq \sigma_j \). Note that this function is well defined exactly when \( \pi \lhd \sigma \). Next, we let \( \| \pi \| \) denote the number of partition elements of \( \pi \). Combining these notations, if \( \pi \lhd \sigma \)  then we let \( \| \sigma \|_\pi \coloneqq \sum_{i = 1}^{\| \pi \|} \sigma_{\pi_i} \). If \( T \subseteq S \), we write \( \pi|_T \) to denote the restriction of the \( \pi \) to the set \( T \) (so that \( \pi|_T \in \mathcal{B}_T \)). If   \( \mu \) is a signed measure on \( (\{ 0,1 \}^S, \sigma(\{0,1\}^S)) \), \( T \subseteq S \) and \( \pi \in \mathcal{B}_T \), we let \( \mu|_T( \pi) \coloneqq \mu \bigl( \{ \pi' \in \mathcal{B_S} \colon \pi'|_T = \pi \} \bigr) \).
If \( \pi',\pi'' \in \mathcal{B}_S \), then we write \( \pi' \lhd \pi'' \) if for each partition element of \( \pi' \) is a subset of some partition element of \( \pi'' \).

We let \( S_n \) denote the set of all permutations of \( [n] \). \( S_n \) acts naturally on \( \mathcal{B}_n \) by permuting the elements in \( [n] \).  When \( \tau \in S_n \) and \( \pi = (\pi_1, \ldots, \pi_m)\in \mathcal{B}_n \), we let \( \tau \circ \pi \coloneqq (\tau(\pi_1), \ldots, \tau(\pi_m))\). If \( \mu \) is a signed measure on \( (\{ 0,1 \}^n, \sigma(\{0,1\}^n)) \) which is such that \( \mu(\tau \circ \pi) = \mu(\pi) \) for all \( \tau \in S_n \) and \( \pi \in \mathcal{B}_n \), we say that \( \mu \) is \emph{permutation invariant}.

Finally, recall when  \( G = (V,E) \) is a finite graph and  \( w \in \{ 0,1 \}^E \), we define  \( E_w \coloneqq \{ e \in E \colon w_e = 1 \}\) and note that this defines a partition \( \pi[w] \in \mathcal{B}_V\) if we let \( v ,v' \in V \) be in the same partition element of \( \pi[w] \) if and only if they are in the same connected component of the graph \( (V,E_w) \). If \( T \subseteq V \) and \( |T| \geq 2 \), then we let \( \pi[T] \) be the unique partition in \( \in \mathcal{B}_V\) in which \( T \) is a partition element and all other partition elements are singletons.

\subsection{The associated linear operator}
Let \( n \in  \mathbb{N} \),  \( \nu \in \mathcal{P}(\{ 0,1 \}^n) \) and
%Let \( X = (X_i)_{i \in [n]}\) be a \( \{ 0,1 \} \)-valued random vector and let \( \nu  \) denote the probability density function of \( X \). 
 \( p = \nu(1^{\{ 1 \}}) \).
It was observed in~\cite{st2017} that if \( \mu \in \RER_{[n]} \) is such that \( \Phi_p(\mu) = \nu\), then \( \mu \) and \( \nu \) satisfy the following set of linear equations.
\begin{equation}\label{eq: linear system}
	\nu(\sigma) = \sum_{\pi \in \mathcal{B}_n \colon \pi \lhd \sigma} p^{\| \sigma\|_\pi}(1-p)^{\|\pi\| - \| \sigma\|_\pi}\mu(\pi), \quad \sigma \in \{ 0,1 \}^n.
\end{equation}
Moreover, whenever a non-negative measure \( \mu \) on \( (\mathcal{B}_n, \sigma(\mathcal{B}_n))\) satisfies these equations, then \( \mu \in \RER_{[n]} \) and \( \Phi_p(\mu) = \nu \), i.e. then \( \mu \) is a color representation of \( \nu \). A signed measure \( \mu  \) on \( (\mathcal{B}_n, \sigma(\mathcal{B}_n)) \) which satisfies~\eqref{eq: linear system}, but which is not necessarily non-negative, will be called a \emph{formal solution} to~\eqref{eq: linear system}. 
If we for a finite set \( S \) let \( \RER_S^* \) denote the set of signed measures on \( (\mathcal{B}_S,\sigma(B_S) ) \) and \( \mathcal{P}^*(\{ 0,1 \} ) \) denote the set of signed measures on \( (\{ 0,1 \}^n, \sigma(\{ 0,1 \}^n)) \), then for each \( p \in (0,1) \) we can use~\eqref{eq: linear system} to  extend \( \Phi_p \colon \RER_S \to \mathcal{P}(\{ 0,1 \}^S) \) to a mapping \( \Phi_S^* \colon \RER_S^* \to \mathcal{P}^*(\{ 0,1 \} )  \), whose restriction to \(\RER_S \) is equal to \( \Phi \).

The matrix corresponding to the system of linear equations given in~\eqref{eq: linear system} is given by
\begin{equation}
	A_{n,p}(\sigma, \pi)  \coloneqq \begin{cases}
	 p^{\| \sigma\|_\pi}(1-p)^{\|\pi\| - \| \sigma\|_\pi} & \text{if } \pi \lhd \sigma \cr
	 0 &\text{else}
	\end{cases}, \quad  \sigma \in \{ 0,1 \}^n,\,  \pi \in \mathcal{B}_n .
\end{equation}
It was shown in~\cite{fs2019} that \( A_{n,1/2} \) has rank \( 2^{n-1}\), and that when \( p \in (0,1) \backslash \{ 1/2 \} \), then \( A_{n,p} \) has rank \( 2^n-n \). When we use the matrix \( A_{n,p} \) to think about~\eqref{eq: linear system} as a system of linear equations, we will abuse notation slightly and let \( \mu \in \RER_{[n]}^*\)  denote both the signed measure and the corresponding vector \( (\mu(\pi))_{\pi \in \mathcal{B}_n} \), given some unspecified and arbitrary ordering of \( \mathcal{B}_n \).

\subsection{Subsequential limits of color representations}

Assume that \( n \in \mathbb{N} \) and that a family \(  \mathcal{N} = (\nu_p)_{p \in (0,1)} \) of probability measures on \( \mathcal{P}(\{ 0,1 \}^n) \) are given. Further, assume that for each \( p \in (0,1 ) \), the marginal distribution of \( \nu_p \) is given by \( (1-p) \delta_0 + p\delta_1 \). We say that a measure \( {\mu \in RER_{[n]}}\) arise as a subsequential limit of color representations of measures in \( \mathcal{N} \) as \({ p \to 1/2} \), if there is a sequence \( p_1, p_2, \ldots \) in \( (0,1)\backslash \{ 1/2 \} \) with \( \lim_{j \to \infty} p_j = 1/2 \) and a measure  \( \mu_{j} \in \Phi_{p_j}^{-1}(\nu_j) \) such that for all \( \pi \in \mathcal{B}_n \) we have \( \lim_{j \to \infty} \mu_j(\pi) =  \mu(\pi) \).

\section{Color representations \texorpdfstring{of \( X^{G, \beta,0} \)}{with no external field}}\label{sec: no external field}

In this section we  give a proof of Theorem~\ref{theorem: color representations with no external field}.  

\begin{proof}[Proof of Theorem~\ref{theorem: color representations with no external field}]

When \( p = 1/2 \), \( \sigma \in \{ 0,1 \}^n \) and \( \pi \in \mathcal{B}_n \), then
\begin{equation}\label{eq: main system}
A(\sigma, \pi) \coloneqq A_{n,p} (\sigma,\pi) =  \begin{cases}
2^{-\| \pi\| } &\textnormal{if } \pi \lhd \sigma ,\cr
0 &\textnormal{ otherwise.}
\end{cases}
\end{equation}
For \( S \subseteq [n] \) and \( \pi \in \mathcal{B}_n \) define
\begin{equation}
A' (S,\pi)\coloneqq  \sum_{\sigma \in \{ 0,1 \}^n \colon \sigma|_S \equiv 1} A(\sigma, \pi)
= 
2^{-\| \pi|_S\| }.
\end{equation}
Since   for any \( S \subseteq [n] \) and \( \pi \in \mathcal{B}_n \) we have 
\begin{align*}
	&
	\sum_{T \subseteq [n] \colon   S \subseteq T} A'(T,\pi) (-1)^{|T|-|S|} 
	=\sum_{T \subseteq [n] \colon \atop S \subseteq T} \sum_{\sigma \in \{ 0,1 \}^n \colon \atop\sigma|_T \equiv 1} A(\sigma, \pi)(-1)^{|T|-|S|} 
	\\&\qquad 
	=\sum_{\sigma \in \{ 0,1 \}^n \colon \atop\sigma|_S \equiv 1}A(\sigma, \pi)
	\sum_{T \subseteq [n] \colon\atop \sigma|_T \equiv 1} (-1)^{|T|-|S|} 
	=
	A(1^S0^{[n]\backslash S},\pi) 
\end{align*}
it follows that \( A \) and \( A' \) are row equivalent. Moreover, by M\"obius inversion theorem, applied to the set of subsets of \( [n] \) ordered by inclusion, the matrix
\begin{equation}
A'' (S,\pi)\coloneqq  \sum_{S' \colon S' \subseteq S} 2^{|S'|}(-1)^{|S|-|S'|} A'(S',\pi), \quad S \subseteq [n],\, \pi \in \mathcal{B}_n
\end{equation}
is   row equivalent to \( A' \), and hence also to \( A \).
By Theorem~1.2~in~\cite{fs2019}, \( A \)   has rank \( 2^{n-1} \), and hence the same is true for \( A'' \).

Now note that if \( S \subseteq [n] \),  \( \pi \in \mathcal{B}_n \), and we  let \( T_1 \), \( T_2\), \ldots, \( T_{\| \pi|_S \|}\) denote the partition elements of \( \pi|_S \), then
\begin{align*}
&\sum_{S'   \colon S' \subseteq S} 2^{|S'|}(-1)^{|S|-|S'|} A'(S',\pi) 
=
2^{|S|}\sum_{S'  \colon S' \subseteq S} (-2)^{|S'|-|S|} A'(S',\pi) 
\\&\qquad = 2^{|S|} \sum_{S'  \colon S' \subseteq S} (-2^{-1})^{|S| - |S'|} \cdot 2^{-\| \pi|_{S'}\| } 
\\&\qquad  = 2^{|S|}\sum_
{\substack{S_1, \ldots, S_m  \colon \\  \forall i \in [m] \colon S_i \subseteq T_i }} \prod_{i=1}^m (-2^{-1})^{|T_i| - |S_i|} 2^{-\mathbb{1}_{S_i \not = \emptyset}}
\\&\qquad  = 2^{|S|}\prod_{i=1}^m \sum_{S_i \colon S_i \subseteq T_i}  (-2^{-1})^{|T_i| - |S_i|} \cdot  2^{-\mathbb{1}_{S_i \not = \emptyset}}
%\\&\qquad = \prod_{i=1}^m ((1-2^{-1})^{|T_i|} \cdot 2^{-1} - (-2^{-1})^{|T_i|} \cdot 2^{-1} + (-2^{-1})^{|T_i|})
 \\&\qquad =2^{|S|} \prod_{i=1}^m (1 + (-1)^{|T_i|} ) \cdot (2^{-1})^{|T_i|+1}
%\\&\qquad = \prod_i 2^{-|T_i|} \, I(|T_i| \text{ is even}) 
 \\&\qquad =  \mathbb{1} (\pi|_S \text{ has only even sized partition elements}).
\end{align*}
and hence 
\begin{equation}
	A''(S,\pi) = \mathbb{1}(\pi|_S \text{ has only even sized partition elements}).
\end{equation}

Let \( T \) be a spanning tree of \( G \). Let \( \mathcal{B}_n^{T} \subseteq \mathcal{B}_n \) denote the partitions of \( [n] \) whose partition elements induce connected subgraphs of \( T \). Note that the number of such partitions is equal to \( 2^{n-1}\). For \( S \subseteq [n] \) with \( |S| \) even and \( \pi \in \mathcal{B}_n^T \),  define
\begin{equation}
	A_T(S,\pi) \coloneqq  \mathbb{1} (\pi|_S \text{ has only even sized partition elements}).
\end{equation}
Then \( A_T \) is a submatrix of \( A'' \). We will  show that \( A_T \) has full rank. Since \( A_T \) is a \( 2^{n-1} \) by \( 2^{n-1} \) matrix, this is equivalent to having non-zero determinant. 
To see that \( \det A_T \not = 0 \), note first that if \( S \subseteq [n] \), \( |S| \) is even and \( \pi \in \mathcal{B}_n^T \), then
\begin{align*}
  B(S, \pi) &\coloneqq  \sum_{\pi' \colon \pi' \lhd \pi} (-1)^{|\pi| - |\pi'|} A_T(S,\pi') 
  \\& = \mathbb{1} \left (\begin{matrix}\pi \text{ has only even sized partition elements}\\ \text{and any finer partition of } S \text{ has at least } \\ \text{one odd sized partition element}  \end{matrix} \right).
\end{align*}
Since all partition elements of \( \pi \in \mathcal{B}_n^T \) induce connected subgraphs of \( G \), \( B\) is a permutation matrix. Since all permutation matrices have non-zero determinant, this implies that \( B \), and hence also \( A_T \), has full rank.
% 
%(To see that \( B \) is indeed a permutation matrix, suppose first that  \( T \) is a line graph. Given \( \pi \in \mathcal{B}_n \), let \( S_\pi \) be the unique set which contains the endpoints of each partition element of \( \pi \) which contains at least two elements of \( S \).  Then  \( B(S, \pi) = I(S = S_\pi)\).)
 %
Since \( A_T \) has \( 2^{n-1} \) rows and columns, this implies in particular that \( A_T \) has rank \( 2^{n-1 } \). On the other hand,  \( A_T \) is a submatrix of \( A'' \), and \( A'' \) is row equivalent to \( A \) which also has rank \( 2^{n-1} \). This implies in particular that when we solve~\eqref{eq: linear system}, we can use the columns corresponding to partitions in \(   \mathcal{B}_n^T \)  as dependent variables. 

Now recall that since \( X^{G, \beta,0}\) is the Ising model on some graph \( G \), \( X^{G, \beta,0} \) has at least one color representation given by \( \mu_{G,1-e^{-2\beta}} \). The random cluster model \( \mu_{G,1-e^{-2\beta}} \) gives strictly positive mass to all partitions \( \pi \in \mathcal{B}_n \) whose partition elements induce connected subgraphs of \( G \). In particular, it gives strictly positive mass to all partitions in \(  \mathcal{B}_n^T \). If we use the columns corresponding to partitions in \(   \mathcal{B}_n^T \)  as dependent variables, then all dependent variables are given positive mass by \( \mu_{G, 1-e^{-2\beta}} \). Since \( n \geq 3 \), there is at least one free variable. By continuity, it follows that we can find another color representation by increasing the value of this free variable a little. If \( G \) is not a tree, there will be at least of free variable corresponding to a partition \( \pi \in \mathcal{B}_n \backslash \mathcal{B}^T \), whose partition elements induce connected subgraphs of \( G \) (but not \( T \)). From this the desired conclusion follows.
\end{proof}

\begin{remark}
It is not the case that all sets of \( 2^{n-1} \) columns of \( A_{n,1/2} \) have full rank. To see this, note first  that there are exactly \( |\mathcal{B}_n| - |\mathcal{B}_{n-1}| \) partitions in \( \mathcal{B}_n \) in which \( 1 \) is not a singleton. If  \( |\mathcal{B}_n| - |\mathcal{B}_{n-1}| \geq 2^{n-1} \), then there is a set \( \mathcal{B}' \subseteq \mathcal{B}_n \) of such partitions of size \( 2^{n-1} \). An easy calculation shows that this happens whenever \( n \geq 4 \).
 Let \( \mu \in \RER_{[n]}^*\) be a signed measure on \( \mathcal{B}_n \) with support only on \(\mathcal{B}'\), and let \( \nu \coloneqq \Phi_{1/2}^*(\mu) \). Then by definition, if \( \nu \in \mathcal{P}(\{ 0,1 \}^n)\) then  
 if \( \nu(1^{\{ 1 \}} 0^{[n] \backslash \{ 1 \}})=0 \). In particular, this implies that the columns of \( A_{n,1/2} \) corresponding to the partitions in \( \mathcal{B}' \) cannot have full rank. 
  
\end{remark}

\section{Color representations of \texorpdfstring{\( X^{K_n, \beta,h} \) for \( n \in \{ 3,4,5\} \)}{the Ising model on K3, K4 and K5}}\label{section: small complete graphs}

In this section we provide proofs of Proposition~\ref{theorem: Ising on K3} and~\ref{theorem: Ising on K4}. In both cases, Mathematica was used to simplify the formulas for the color representations for different values of \( \beta \) and \( h \).

\begin{proof}[Proof Proposition~\ref{theorem: Ising on K3}]
Fix  \( h > 0 \)  and let \( p  =  \nu_{K_3, \beta,h}(1^{\{1\}})  \). Note that since \( |V(K_3)| = 3 \), the relevant set of partitions is given by 
\[ \mathcal{B}_3 = \Bigl\{ \bigl(\{1,2,3\}\bigr),\, \bigl(\{1,2\},\{3\}\bigr),\, \bigl(\{1\},\{2,3\}\bigr),\, \bigl(\{1,3\},\{2\}\bigr),\, \bigl(\{1\},\{2\},\{3\}\bigr)\Bigr\}. \]
By Theorem 1.4 in~\cite{fs2019}, the linear equation system \( A_{3,p} \mu = \nu_{K_3,\beta, h} \) has a unique formal solution \( \mu \in RER_{[3]}^*\). With some work, one verifies that this solution satisfies 
\[
 \mu\bigl((\{1,2\},\{3\})\bigr) = \mu\bigl((\{1,3\},\{2\})\bigr) = \mu\bigl((\{1\},\{2,3\})\bigr),
 \]
  and
\[
\mu\bigl((\{1,2,3\})\bigr) = \frac{\left(e^{4\beta }-1\right)^2    e^{4 h} \left(e^{4 \beta }+e^{4 \beta +2 h}+e^{4 \beta +4 h}+5 e^{2 h}+2 e^{4 h}+2\right)}{\left(e^{4 \beta }+2 e^{2 h}+e^{4 h}\right) \left(e^{4 \beta +4 h}+2 e^{2 h}+1\right) \left(e^{4 \beta }+e^{4 \beta +2 h}+e^{4 \beta +4 h}+e^{2 h}\right)}
\]
\[
\mu\bigl((\{1\},\{2,3\})\bigr) = \frac{\left(e^{4\beta }-1\right)   e^{2 h} \left(e^{2 h}+1\right)^2 \left(e^{4 \beta }-e^{4 \beta +2 h}+e^{4 \beta +4 h}+3 e^{2 h}\right)}{\left(e^{4 \beta }+2 e^{2 h}+e^{4 h}\right) \left(e^{4 \beta +4 h}+2 e^{2 h}+1\right) \left(e^{4 \beta }+e^{4 \beta +2 h}+e^{4 \beta +4 h}+e^{2 h}\right)}
\]
\[
\mu\bigl((\{1\},\{2\},\{3\})\bigr) =\frac{\left(e^{2 h}+1\right)^2 \left(e^{4 \beta }-e^{4 \beta +2 h}+e^{4 \beta +4 h}+3 e^{2 h}\right)^2}{\left(e^{4 \beta }+2 e^{2 h}+e^{4 h}\right) \left(e^{4 \beta +4 h}+2 e^{2 h}+1\right) \left(e^{4 \beta }+e^{4 \beta +2 h}+e^{4 \beta +4 h}+e^{2 h}\right)}.
\]
Since \( e^{4 \beta +4 h} \geq e^{4 \beta +2 h} \), it is immediately clear that for all \( \pi \in \mathcal{B}_3 \), \( \beta > 0 \) and \( h > 0 \),  \( \mu(\pi) \) is non-negative, and hence (i) holds.

By letting \( h \to 0 \) in the expression for \( \mu\bigl((\{1\},\{2\},\{3\})\bigr) \), while keeping \( \beta \) fixed, one obtains 
\[
 \lim_{h \to 0 } \mu\bigl((\{1\},\{2\},\{3\})\bigr)  = \frac{4}{1+3e^{4\beta}}.
 \]
On the other hand, it is easy to check that  
 \[
\mu_{K_3,e^{-2\beta}}\bigl((\{1\},\{2\},\{3\})\bigr) = %\frac{(1-p)^3 2^3}{(1-p)^32^3 + 3p(1-p)^2 2^2 + 3p^2(1-p) 2^1 + p^3 2^1} = 
 \frac{4 e^{-2 \beta }}{3+e^{4 \beta }}.
% \frac{4 e^{12 \beta}}{1+3 e^{8 \beta}}.
 \]
Since these two expressions are not equal for any \( \beta > 0 \), (ii) holds.
\end{proof}

\begin{proof}[Proof of Proposition~\ref{theorem: Ising on K4}]
Fix  \( h > 0 \) and let \( p  \coloneqq  \nu_{K_4,\beta,h} (1^{\{ 1 \}}) \).
Since \( \nu_{K_4,\beta,h} \) is permutation invariant, it follows that if   \( \nu_{K_4, \beta,h} \) has a color representation, then it has at least one  color representation which is invariant under the action of \( S_4 \). It is easy to check that there are exactly five different partitions in \( \mathcal{B}_4 /S_4 \), namely \( (\{1,2,3,4\}) \), \( ( \{1,2,3\},\{4\} ) \), \( (\{1,2\},\{3,4\}) \), \( (\{1\},\{2\},\{3,4\}) \) and \( (\{1\},\{2\},\{3\},\{4\}) \).
By Theorem~1.5(i) in~\cite{fs2019}, the linear subspace spanned by all \( S_4 \)-invariant formal solutions \( \mu \) of \( A_{4,p} \mu = \nu_{K_4,\beta,h} \) has dimension one. 
By linearity, this implies in particular that if \( \nu_{K_4, \beta,h} \) has a  color representation, then it has at least one color representation which is \( S_4 \)-invariant and gives zero mass to at least one of the partitions in \( \mathcal{B}_4 /S_4 \). This gives us one unique solution for each choice of partition in \( \mathcal{B}_4 /S_4 \).

Solving the corresponding linear equation systems in Mathematica and studying the solutions, after some work, one obtains the desired necessary and sufficient condition.

\end{proof}

 \begin{remark}\label{remark: K5}
With the same strategy as in the proof of Proposition~\ref{theorem: Ising on K4}, one can find analogous necessary and sufficient conditions for \( G = K_n \) when \( n\geq 5 \). 
However, already when \( n = 5 \) the analogue inequality is quite complicated; 
in this case, one can show that a necessary and sufficient condition to have a color representation is that
\[
\begin{split}
-x^{18}  y^{10} &-x^{18} y^8+x^{18} y^7-x^{18} y^6-x^{18} y^4+x^{16} y^{14}+x^{16} y^{12}+x^{16} y^9
\\&+3 x^{16} y^8+2 x^{16} y^7+3 x^{16} y^6+x^{16} y^5+x^{16} y^2+x^{16}-3 x^{14} y^{11}
\\& -x^{14} y^{10}-12 x^{14} y^9+6 x^{14} y^8-11 x^{14} y^7+6 x^{14} y^6-12 x^{14} y^5-x^{14} y^4
\\&-3 x^{14} y^3+9 x^{12} y^{13}-3 x^{12} y^{12}+6 x^{12} y^{11}+7 x^{12} y^{10}+15 x^{12} y^9-4 x^{12} y^8
\\&+18 x^{12} y^7-4 x^{12} y^6+15 x^{12} y^5+7 x^{12} y^4+6 x^{12} y^3-3 x^{12} y^2+9 x^{12} y
\\&+19 x^{10} y^{12}+27 x^{10} y^{11}+14 x^{10} y^{10}+36 x^{10} y^9-21 x^{10} y^8+45 x^{10} y^7-21 x^{10} y^6
\\&+36 x^{10} y^5+14 x^{10} y^4+27 x^{10} y^3+19 x^{10} y^2+20 x^{8} y^{12}-18 x^{8} y^{11}+15 x^{8} y^{10}
\\&-11 x^{8} y^9+50 x^{8} y^8-102 x^{8} y^7+50 x^{8} y^6-11 x^{8} y^5+15 x^{8} y^4-18 x^{8} y^3
\\&+20 x^{8} y^2+82 x^{6} y^{11}+83 x^{6} y^{10}+76 x^{6} y^9+178 x^{6} y^8+197 x^{6} y^7+178 x^{6} y^6
\\&+76 x^{6} y^5+83 x^{6} y^4+82 x^{6} y^3+54 x^4 y^{10}+226 x^4 y^9+152 x^4 y^8+386 x^4 y^7
\\&+152 x^4 y^6+226 x^4 y^5+54 x^4 y^4-84 x^2 y^9+12 x^2 y^8-156 x^2 y^7+12 x^2 y^6
\\&-84 x^2 y^5-72 y^8-56 y^7-72 y^6
\geq 0
\end{split}
\]
where \( x = e^{2\beta} \) and \( y = e^{2h} \).

%One can verify that for \( h \) sufficiently small, this gives a phase transition at the largest positive root to 
%\[
%\beta_0 = \frac{\log ( 5  +  \sqrt{17}) - \log 2}{16}
%\] 
%Finally, one can also verify that for large \( h \) and \( \beta \), we asymptotically need to have that \( h > \beta \).

\end{remark}

\section{Color representations for small \texorpdfstring{\( h>0 \)}{h>0}}\label{section: small h}

\subsection{Existence}
The goal of this subsection is to provide a proof of Theorem~\ref{theorem: no cr when critical or supercritical}. A main tool in the proof of this theorem will be the following result from~\cite{fs2019}.

\begin{theorem}[Theorem 1.7 in~\cite{fs2019}]\label{theorem: lim sol I}
Let \( n \in  \mathbb{N} \) and let \( (\nu_p)_{p \in (0,1)} \) be a family of probability measures on \( \{ 0,1\}^n \). 
For each \( p \in (0,1 ) \), assume that \( \nu_p \) has marginals   \(  p\delta_1 + (1-p) \delta_0 \),  and that for each \( S \subseteq [n] \), \( \nu_p(1^S) \) is two times differentiable in \( p \) at \( p = 1/2 \). Assume further that for any \( S \subseteq T \subseteq [n] \) and any \( p \in (0,1) \), we have that
\[
\nu_p (0^S1^{T \backslash S}) = \nu_{1-p} (0^{T \backslash S} 1^S).
\]
Then the set of solutions \( \{ (\mu(\pi))_{\pi \in \mathcal{B}_n} \} \) to~\eqref{eq: main system} which arise as subsequential limits of solutions to~\eqref{eq: linear system} as \( p \to 1/2 \) is exactly the set of solutions to the system of equations
%\begin{align*}
%&\sum_\sigma      I(\sigma_S\text{ has at most one odd sized partition element } T_{\sigma_S}) 
% \,      q_\sigma
%\\[-1ex]&\qquad =
%\begin{cases}
% \sum_{S' \subseteq S}  (-2)^{|S'|-1} \nu'(1^{S'})  & \text{if } |S| \text{ is odd} \cr
% \sum_{S' \subseteq S}  (-2)^{ |S'|} \nu(1^{S'}) & \text{if } |S| \text{ is even.}
%\end{cases}
%\end{align*}
\begin{equation}\label{eq: limiting system}
\begin{cases}
\sum_{\pi \in \mathcal{B}_n}  2^{-\| \pi\| } \mu(\pi) = \nu_{1/2}(1^S), & S \subseteq [n],\, |S| \text{ even}  \cr
\sum_{\pi \in \mathcal{B}_n} \| \pi \|_S 2^{-\| \pi|_S\|+1 } \mu(\pi) = \nu'_{1/2}(1^S), & S \subseteq [n],\, |S| \text{ odd.} 
\end{cases}
\end{equation}

\end{theorem}

\begin{remark}\label{remark: better form of equations system}
	By applying elementary row operations to the system of linear equations in~\eqref{eq: limiting system}, one obtains the following equivalent system of linear equations.
\begin{equation}\label{eq: alternative form}
	\begin{split}
&\sum_{\pi \in \mathcal{B}_n}      \mathbb{1} (\pi|_S\text{ has at most one odd sized partition element}) 
 \,      \mu(\pi)
\\[-1ex]&\qquad =
\begin{cases}
 \sum_{S' \subseteq S}  (-2)^{|S'|+1} \nu_{1/2}'(1^{S'})  & \text{if } |S| \text{ is odd} \cr
 \sum_{S' \subseteq S}  (-2)^{ |S'|} \nu_{1/2}(1^{S'}) & \text{if } |S| \text{ is even.}
\end{cases}
\end{split}
\end{equation}
(See the last equation in the proof of Theorem 1.7 in~\cite{fs2019}.)
\end{remark}

The proof of Theorem~\ref{theorem: no cr when critical or supercritical} will be divided into two parts, corresponding to the two cases \( \hat \beta > 1 \) and \( \hat \beta = 1 \).

\begin{proof}[Proof of Theorem~\ref{theorem: no cr when critical or supercritical} when \( \hat \beta > 1 \) (the supercritical regime).]

%Assume that \( {n \in \mathbb{N}} \) is large and let \( X^{K_n, \beta,0}\) be the Ising model on \( K_n \) with parameter \( \beta > 0 \) and external field \( h = 0 \).

%
Let \( z_{\hat \beta} \) be the unique positive root to the equation \( z = \tanh ({\hat \beta} z) \). Then it is well known (see e.g.~\cite{e1985}, p.\ 181) that as \( n \to \infty \), we have
\[
 \frac{ 2\| X^{K_n, \beta,0} \|-n}{n} \overset{d}{\to} \frac{1}{2} \delta_{-z_{\hat \beta}} + \frac{1}{2} \delta_{z_{\hat \beta}} 
 \]
 and hence, by Lemma~\ref{lemma: right hand side components}, as \( n \to \infty \) we have that
\begin{equation}\label{eq: expression limits}
\begin{cases}
\sum_{\sigma \in \{ 0,1 \}^n} \chi_\emptyset(\sigma)  \, \nu_{K_n, \beta,0}(\sigma) = 1 \cr
\sum_{\sigma\in \{ 0,1 \}^n} (2\| \sigma \|-n) \chi_{\{ 1 \}}(\sigma)\, \nu_{K_n, \beta,0}(\sigma) \sim  n z_{\hat \beta}^2 \cr
\sum_{\sigma\in \{ 0,1 \}^n} \chi_{[2]}(\sigma) \, \nu_{K_n, \beta,0}(\sigma)\sim z_{\hat \beta}^2 \cr
\sum_{\sigma\in \{ 0,1 \}^n} (2\| \sigma \|-n) \chi_{[3]}(\sigma) \, \nu_{K_n, \beta,0}(\sigma)\sim n z_{\hat \beta}^4 \cr
\sum_{\sigma\in \{ 0,1 \}^n} \chi_{[4]}(\sigma) \, \nu_{K_n, \beta,0}(\sigma)  \sim z_{\hat \beta}^4.
\end{cases}
\end{equation}

By Remark~\ref{remark: better form of equations system}, any color representation \( \mu \) of \( \nu_{K_n,\beta,0} \) which arise as a subsequential limit of color representations of \( \nu_{K_n,\beta,h} \) as \( h \to 0 \) must satisfy~\eqref{eq: alternative form}.
By Lemma~\ref{lemma: transitive}, for any \( S \subseteq [n] \) we have that 
\[ \sum_{S' \subseteq S} (-2)^{|S|} \nu_{K_n,\beta,0}(1^{S'}) = (-1)^{|S|}  \sum_{\sigma \in \{ 0,1 \}^n}   \chi_{S}(\sigma) \, \nu_{K_n,\beta,0}(\sigma) 
\]
and
\[ \sum_{S' \subseteq S} (-2)^{|S|+1} \nu_{K_n,\beta,0}'(1^{S'}) = (-1)^{|S|+1}   \,  \frac{ \sum_{\sigma \in \{ 0,1 \}^n} (2\| \sigma \|-n) \chi_{S}(\sigma)  \,   \nu_{K_n,\beta,0}(\sigma)}{\sum_{\sigma \in \{ 0,1 \}^n} (2\| \sigma \|-n) \chi_{\{ 1 \}}(\sigma)  \,   \nu_{K_n,\beta,0}(\sigma)}
\]
and hence~\eqref{eq: alternative form} is equivalent to that
\begin{equation}\label{eq: alternative form ii}
	\begin{split}
&\sum_{\pi \in \mathcal{B}_n}      \mathbb{1}(\pi|_S\text{ has at most one odd sized partition element}) 
 \,      \mu(\pi)
\\[-1ex]&\qquad =
\lambda_S^{(n)} \coloneqq 
\begin{cases}
\frac{ \sum_{\sigma \in \{ 0,1 \}^n} (2\| \sigma \|-n) \chi_{S}(\sigma)   \,   \nu_{K_n,\beta,0}(\sigma)}{\sum_{\sigma \in \{ 0,1 \}^n} (2\| \sigma \|-n) \chi_{\{ 1 \}}(\sigma)  \,   \nu_{K_n,\beta,0}(\sigma)} & \text{if } |S| \text{ is odd} \cr
\sum_{\sigma \in \{ 0,1 \}^n}  \chi_{S}(\sigma)  \, \nu_{K_n,\beta,0}(\sigma) & \text{if } |S| \text{ is even.}
\end{cases}
\end{split}
\end{equation}

Note that since \( \nu_{K_n,\beta,0} \) is permutation invariant, we have \( \lambda_S^{(n)} = \lambda^{(n)}_{[|S|]} \) for any \( S \subseteq [n] \). This implies in particular that if~\eqref{eq: alternative form ii} has a non-negative solution, then it must have at least one permutation invariant non-negative solution \( \mu^{(n)} \).  Observe that \( \mu^{(n)}|_{\mathcal{B}_4} \) satisfies~\eqref{eq: alternative form ii} for any set \( S \subseteq \{ 1,2,3,4 \} \).
For this reason, for each \( n \geq 4 \), we now consider the system of linear equations given by
\begin{equation}\label{eq: alternative form iii} 
\begin{cases}
\sum_{\pi \in \mathcal{B}_4}      \mathbb{1}(\pi|_S\text{ has at most one odd sized partition element}) 
 \,      \mu(\pi)
 =
\lambda_S^{(n)}  \cr
\mu(\pi) = \mu(\tau \circ\pi) \text{ for all } \tau \in S_n
\end{cases},\quad S \subseteq [4].
\end{equation}
Let \( A^{(4)} \) denote the corresponding matrix. By Theorem 1.5(i) in~\cite{fs2019}, the null space of \( A^{(4)} \) has dimension one.
Since~\eqref{eq: alternative form iii} is a set of linear equations, this implies that if a non-negative solution exists, then there will  exist a non-negative solution \( \mu \) in which either \( \mu\bigl((\{1,2,3,4\})\bigr) \), \( \mu\bigl((\{1,2,3\},\{4\})\bigr)\), \( \mu\bigl((\{1,2\},\{3,4\})\bigr)\), \( \mu\bigl((\{1,2\},\{3\},\{4\} )\bigr) \) or \( \mu\bigl((\{1\},\{2\},\{3\},\{4\})\bigr)\) is equal to zero. 
Next, note that when  \( S \subseteq [n] \) satisfies  \( |S| \leq 4 \), by~\eqref{eq: expression limits}, we have that \( \lambda^{(\infty)}_S \coloneqq \lim_{n \to \infty} \lambda^{(n)}_S = z_{\hat \beta}^{2 \lfloor |S|/2 \rfloor} \). 
Define \( \lambda^{(n)} = (\lambda^{(n)}_{S})_{S \subseteq [4]} \) and \( \lambda^{(\infty)} = (\lambda^{(\infty)}_{S})_{S \subseteq [4]} \).
Using these definitions,  one verifies that the five (permutation invariant) solutions \( \mu_1^{(\infty)}\), \( \mu_2^{(\infty)}\), \( \mu_3^{(\infty)}\), \( \mu_4^{(\infty)}\) and \( \mu_5^{(\infty)}\) to
\begin{equation} \label{eq: large n system}
	A^{(4)} \mu  = \lambda^{(\infty)}
\end{equation} 
with at least one zero entry are given by
\begin{align*}
&\mu_1^{(\infty)}\bigl({(\{1,2,3,4\}),\, (\{1\},\{2,3,4\}),\, (\{1,2\},\{3,4\}),\, (\{1\},\{2\},\{3,4\}),\, (\{1\},\{2\},\{3\},\{4\})}\bigr)
\\&\qquad = 
\bigl(0,z_{\hat \beta}^2,\, z_{\hat \beta}^4/3,\, \textcolor{red}{-z_{\hat \beta}^2(1+z_{\hat \beta}^2/3)},\, (1+z_{\hat \beta}^2)^2\bigr)  
\end{align*}
\begin{align*}
&\mu_2^{(\infty)}\bigl({(\{1,2,3,4\}),\, (\{1\},\{2,3,4\}),\, (\{1,2\},\{3,4\}),\, (\{1\},\{2\},\{3,4\}),\, (\{1\},\{2\},\{3\},\{4\})}\bigr)
\\&\qquad = 
\bigl(z_{\hat \beta}^2,\, 0,\, z_{\hat \beta}^2(z_{\hat \beta}^2-1)/3,\,  \textcolor{red}{-z_{\hat \beta}^2(z_{\hat \beta}^2-1)/3},\,  (z_{\hat \beta}^2-1)^2\bigr) 
\end{align*}

\begin{align*}
&\mu_3^{(\infty)}\bigl({(\{1,2,3,4\}),\, (\{1\},\{2,3,4\}),\, (\{1,2\},\{3,4\}),\, (\{1\},\{2\},\{3,4\}),\, (\{1\},\{2\},\{3\},\{4\})}\bigr)
\\&\qquad =
 \bigl(z_{\hat \beta}^4,\, \textcolor{red}{-z_{\hat \beta}^2(z_{\hat \beta}^2-1)},\, 0,\, z_{\hat \beta}^2(z_{\hat \beta}^2-1),\textcolor{red}{-(z_{\hat \beta}^2-1)(1+3z_{\hat \beta}^2)}\bigr) 
\end{align*}

\begin{align*}
&\mu_4^{(\infty)}\bigl({(\{1,2,3,4\}),\, (\{1\},\{2,3,4\}),\, (\{1,2\},\{3,4\}),\, (\{1\},\{2\},\{3,4\}),\, (\{1\},\{2\},\{3\},\{4\})}\bigr)
\\&\qquad = 
\bigl(z_{\hat \beta}^2 (3+z_{\hat \beta}^2)/4, \, \textcolor{red}{-z_{\hat \beta}^2(z_{\hat \beta}^2-1)/4},\, z_{\hat \beta}^2(z_{\hat \beta}^2-1)/4,\, 0,\, \textcolor{red}{-(z_{\hat \beta}^2-1)}\bigr)  
\end{align*}
and
\begin{align*}
&\mu_5^{(\infty)}\bigl({(\{1,2,3,4\}),\, (\{1\},\{2,3,4\}),\, (\{1,2\},\{3,4\}),\, (\{1\},\{2\},\{3,4\}),\, (\{1\},\{2\},\{3\},\{4\})}\bigr)
\\&\qquad = 
\bigl((1+z_{\hat \beta}^2)^2/4,\, \textcolor{red}{-(z_{\hat \beta}^2-1)^2/4},\, (z_{\hat \beta}^2-1)(1+3z_{\hat \beta}^2)/12,\, \textcolor{red}{-(z_{\hat \beta}^2-1)/3},\, 0\bigr). 
\end{align*}
Above, all the negative entries are displayed in red. Since all of these solutions have at least one strictly negative entry, it follows that any solutions \( \mu^{(\infty)}\) to \( A^{(4)} \mu^{(\infty)} = \lambda^{(\infty)}\) must have at least one strictly negative entry.

Now let \( \mu^{(n)} \) be a permutation invariant solution to~\eqref{eq: alternative form ii}. Then   \( \mu^{(n)}|_{\mathcal{B}_4}\)  must satisfy \( A^{(4)}  \mu^{(n)}|_{\mathcal{B}_4} = \lambda^{(n)} \).
 But this implies that for all \( S \subseteq [4] \) we have that
\begin{align*}
&\lim_{n \to \infty} ( A^{(4)}  \mu^{(n)}|_{\mathcal{B}_4} - A^{(4)}  \mu^{(\infty)})(S) 
= \lim_{n \to \infty} \lambda^{(n)}(S) - \lambda^{(\infty)}(S) = 0.
\end{align*}
Since
\begin{align*}
& \lim_{n \to \infty} ( A^{(4)}  \mu^{(n )}|_{\mathcal{B}_4} - A^{(4)}  \mu^{(\infty)})(S)  
 =
 A^{(4)} ( \lim_{n \to \infty}  \mu^{(n )}|_{\mathcal{B}_4} -  \mu^{(\infty)})(S)  
\end{align*}
this show that
\[
 \lim_{n \to \infty} \mu^{(n )}|_{\mathcal{B}_4} -  \mu^{(\infty )} \in \Null A^{(4)}.
\]
This in particular implies that \( \mu^{(n)} \) must have negative entries for all sufficiently large \(n \), and hence the desired conclusion follows.
\end{proof}

We now give a proof of Theorem~\ref{theorem: no cr when critical or supercritical} in the case \( \hat \beta = 1 \). This proof is very similar to the previous proof, but requires that we look at the distribution of the first five coordinates of \( X_{K_n,\beta,0}\) and are more careful with the asymptotics.

\begin{proof}[Proof of Theorem~\ref{theorem: no cr when critical or supercritical} when \( \hat \beta =1\) (the critical regime).]
Assume that \( n  \) is large.
By Theorem V.9.5 in~\cite{e1985}, as \( n \to \infty \), we have that
\[
 \frac{2\| X^{K_n,\beta,0}\|-n}{n^{3/4}} \overset{d}{\to} X,
 \]
 where \( X \) is a random variable with probability density function \( f(x) = \frac{3^{1/4} \Gamma(1/4)}{\sqrt{2}} e^{-x^4/12},\)  \( x \in \mathbb{R} \).

 Let \( m_1^{(n)}\), \( m_2^{(n)}\), etc. be the moments of \( X^{K_n,\beta,0} \). By Lemma~\ref{lemma: right hand side components}, as \( n \to \infty \), we have that
 \begin{align*}
 \begin{cases}
\sum_{\sigma \in \{ 0,1 \}^n}  \chi_\emptyset (\sigma)  \, \nu_{K_n,\beta,0}(\sigma) = 1 \cr
\sum_{\sigma \in \{ 0,1 \}^n} (2\| \sigma \|-n)  \chi_{\{1\}}(\sigma) \, \nu_{K_n,\beta,0}(\sigma) \sim n^{1/2} m_2^{(n)}\cr
\sum_{\sigma \in \{ 0,1 \}^n}  \chi_{[2]}(\sigma)  \, \nu_{K_n,\beta,0}(\sigma) \sim   n^{-1/2}  m_2^{(n)}   \cr
\sum_{\sigma \in \{ 0,1 \}^n} (2\| \sigma \|-n)  \chi_{[3]}(\sigma)  \, \nu_{K_n,\beta,0}(\sigma) \sim  m_4^{(n)} \cr
\sum_{\sigma \in \{ 0,1 \}^n} \chi_{[4]}(\sigma) \, \nu_{K_n,\beta,0}(\sigma)  \sim  n^{-1} m_4^{(n)}\cr
\sum_{\sigma \in \{ 0,1 \}^n} (2\| \sigma \|-n)  \chi_{[5]}(\sigma)  \, \nu_{K_n,\beta,0}(\sigma) \sim n^{-1/2}   m_6^{(n)} .
\end{cases}
\end{align*}
For each \( S \subseteq [n] \), let \( \lambda^{(n)}_S \) be defined by~\eqref{eq: alternative form ii} and define  \( \lambda^{(n)} = (\lambda^{(n)}_S)_{S \subseteq [n]}\). Further, let \( \hat \lambda^{(n)} \) be the vector containing only the largest order term of each entry of \( \lambda^{(n)}\), i.e.\ let
\begin{equation*}
	\begin{cases}
		\hat\lambda^{(n)}(\emptyset) = 1 \cr
		\hat\lambda^{(n)}([1]) = 1 \cr
		\hat\lambda^{(n)}([2]) = n^{-1/2}m_2^{(n)} \cr
		\hat\lambda^{(n)}([3]) = n^{-1/2} m_4^{(n)}/m_2^{(n)} \cr
		\hat\lambda^{(n)}([4]) = n^{-1}m_4^{(n)} \cr
		\hat\lambda^{(n)}([5]) = n^{-1}m_6^{(n)}/m_2^{(n)}. \cr
	\end{cases}
\end{equation*}
 Then we have that \( \hat \lambda^{(n)}(S) - \lambda^{(n)}(S) = o(n^{-1/2}) \) for all \( S \subseteq [n] \).

We now proceed as in the supercritical case.
By Remark~\ref{remark: better form of equations system}, any color representation \( \mu \) of \( \nu_{K_n,\beta,0} \) which arise as a subsequential limit of color representations of \( \nu_{K_n,\beta,h} \) as \( h \to 0 \) must satisfy~\eqref{eq: alternative form}.
By Lemma~\ref{lemma: transitive}, for any \( S \subseteq [n] \) we have that 
\[ \sum_{S' \subseteq S} (-2)^{|S|} \nu_{K_n,\beta,0}(1^{S'}) = (-1)^{|S|}  \sum_{\sigma \in \{ 0,1 \}^n}   \chi_{S}(\sigma) \, \nu_{K_n,\beta,0}(\sigma) 
\]
and
\[ \sum_{S' \subseteq S} (-2)^{|S|+1} \nu_{K_n,\beta,0}'(1^{S'}) = (-1)^{|S|+1}   \,  \frac{ \sum_{\sigma \in \{ 0,1 \}^n} (2\| \sigma \|-n) \chi_{S}(\sigma)  \,   \nu_{K_n,\beta,0}(\sigma)}{\sum_{\sigma \in \{ 0,1 \}^n} (2\| \sigma \|-n) \chi_{\{ 1 \}}(\sigma)  \,   \nu_{K_n,\beta,0}(\sigma)}
\]
and hence~\eqref{eq: alternative form} is equivalent to that
\begin{equation}\label{eq: alternative form ii}
	\begin{split}
&\sum_{\pi \in \mathcal{B}_n}      \mathbb{1}(\pi|_S\text{ has at most one odd sized partition element}) 
 \,      \mu(\pi)
\\[-1ex]&\qquad =
\lambda_S^{(n)} \coloneqq 
\begin{cases}
\frac{ \sum_{\sigma \in \{ 0,1 \}^n} (2\| \sigma \|-n) \chi_{S}(\sigma)   \,   \nu_{K_n,\beta,0}(\sigma)}{\sum_{\sigma \in \{ 0,1 \}^n} (2\| \sigma \|-n) \chi_{\{ 1 \}}(\sigma)  \,   \nu_{K_n,\beta,0}(\sigma)} & \text{if } |S| \text{ is odd} \cr
\sum_{\sigma \in \{ 0,1 \}^n}  \chi_{S}(\sigma)  \, \nu_{K_n,\beta,0}(\sigma) & \text{if } |S| \text{ is even.}
\end{cases}
\end{split}
\end{equation}
Note that  \( \nu_{K_n,\beta,0} \) is permutation invariant, and hence \( \lambda_S^{(n)} = \lambda^{(n)}_{[|S|]} \) for any \( S \subseteq [n] \). This implies in particular that if~\eqref{eq: alternative form ii} has a non-negative solution, then it must have at least one permutation invariant non-negative solution \( \mu^{(n)} \).  
Observe that \( \mu^{(n)}|_{\mathcal{B}_5} \) satisfies~\eqref{eq: alternative form ii} for any set \( S \subseteq \{ 1,2,3,4, 5 \} \).
For this reason, for each \( n \geq 5 \), we now consider the system of linear equations given by
\begin{equation}\label{eq: alternative form iiii} 
\begin{cases}
\sum_{\pi \in \mathcal{B}_5}      \mathbb{1}(\pi|_S\text{ has at most one odd sized partition element}) 
 \,      \mu(\pi)
 =
\lambda_S^{(n)}  \cr
\mu(\pi) = \mu(\tau \circ\pi) \text{ for all } \tau \in S_n
\end{cases},\quad S \subseteq [5].
\end{equation}
Let \( A^{(5)} \) denote the corresponding matrix and define 
\begin{equation*}
\begin{split}
	&\Pi_5 \coloneqq \Bigl\{ \bigl( \{1,2,3,4,5\} \bigr),\,
	 \bigl(\{1,2,3,4\},\{5\} \bigr),\,
	  \bigl( \{1,2,3\},\{4,5\} \bigr),\,
	   \bigl(\{1,2,3\},\{4\},\{5\} \bigr), \,
	\\&\qquad\qquad \bigl(\{1,2\},\{3,4\},\{5\}\bigr),\,
	 \bigl(\{1,2\},\{3\},\{4\},\{5\}\bigr),\, \bigl(\{1\},\{2\},\{3\},\{4\},\{5\}\bigr) \Bigr\} \subseteq \mathcal{B}_5.
\end{split}	
\end{equation*}
By Theorem 2.2(i) in~\cite{fs2019}, the null space of  \( A^{(5)} \) has dimension two. This implies that if a  non-negative solution \( \mu \) to~\eqref{eq: alternative form iiii} exists, then there will exist two district partitions \( \pi',\pi'' \in \Pi_5 \) and a non-negative solution \( \mu \) to~\eqref{eq: alternative form iiii} such that at \( \mu(\pi') = \mu(\pi'') = 0 \).

 Let \( \mu_1^{(n)}, \mu_2^{(n)} , \ldots, \mu^{(n)}_{\binom{7}{2}} \in RER_{[5]}^*\) be the solutions of~\eqref{eq: alternative form iiii} which are such that at least two of \( \mu\bigl((\{1,2,3,4,5\})\bigr)\), \( \mu\bigl((\{1,2,3,4\},\{5\})\bigr)\), \( \mu\bigl((\{1,2,3\},\{4,5\})\bigr)\), \( \mu\bigl((\{1,2,3\},\{4\},\{5\})\bigr)\), \( \mu\bigl((\{1,2\},\{3,4\},\{5\})\bigr)\), \( \mu\bigl((\{1,2\},\{3\},\{4\},\{5\})\bigr)\) and \\ \( \mu\bigl((\{1\},\{2\},\{3\},\{4\},\{5\})\bigr)\) are equal to zero, and let \( \hat \mu_1^{(n)}, \hat \mu_2^{(n)} , \ldots, \hat \mu^{(n)}_{\binom{7}{2}} \in RER_{[5]}^*\) be the corresponding solutions if we replace \( \lambda^{(n)} \) with \( \hat \lambda^{(n)} \) in~\eqref{eq: alternative form iiii}.
Then one verifies that there is an absolute constant \( C > 0 \) such that for all \( n \geq 5 \) and any  \( i \in \{ 1,2, \ldots, \binom{7}{2} \} \) there is \( \pi \in \mathcal{B}_5 \) such that \( \hat \mu_i^{(n)}(\pi) < -Cn^{-1/2} + o(n^{-1/2}) \).
 Next, note that for any \( i \in \{ 1,2, \ldots, \binom{7}{2} \} \) and \( S \subseteq [5]\), we have
\[
A^{(5)} \mu_i^{(n)}(S) - A^{(5)} \hat \mu_i^{(n)}(S) = \lambda^{(n)}(S) - \hat \lambda^{(n)}(S) = o(n^{-1/2}),
\]
Since \( A^{(5)} \) does not depend on \( n \), it follows that \( \mu_i(S) - \hat \mu_i(S) = o(n^{-1/2}) \) for all \( S \subseteq [5] \),  implying in particular that 
\( \mu_i^{(n)} \) has the same sign as the corresponding entry of \( \hat \mu_i \) at at least one entry which is negative in \( \hat \mu_i \). This gives the desired conclusion.

\end{proof}

\subsection{The random cluster model is almost never a limiting color representation}

In this section, we give a proof of Theorem~\ref{theorem: rcm limit}. %
A main tool in the proof of this result will be the next lemma. The main idea of the proof of Theorem~\ref{theorem: rcm limit} will  be to show that the necessary condition given by this lemma does in fact not hold for the random cluster model.

%Next, we will prove another version of the previous lemma. We included the previous version since it is easier to comprehend, but the following version will be easier to use in computations.

Throughout this section, we will use the following notation.
For a fixed \( n \in  \mathbb{N} \), a set \( S \subseteq [n] \) and a partition \( \pi \in \mathcal{B}_n \), define
\[
A(S,\pi) \coloneqq  \mathbb{1}(\pi|_S\text{ has at most one odd sized partition element}).
\]
If \( |S| \) is odd and  \( A(S,\pi) = 1 \),   let \( T_{S,\pi} \) denote the unique odd sized partition element of \( \pi|_S \).

\begin{lemma}\label{lemma: necessary condition}

Let \( n \in  \mathbb{N} \) and let  \( G  \) be a graph with \( n \) vertices. Let \( \beta > 0 \) and assume  that \( \mu \in \Phi_{1/2}^{-1}(\nu_{G, \beta,0})  \)   arises as a subsequential limit of color representations of \( X^{G, \beta,h} \) as \( h \to 0 \). Finally, let \( S \subseteq [n] \) be such that \( |S| \) is odd. Then
\begin{equation}\label{eq: first version of equality}
n  \biggl[ \, \sum_{\pi \in \mathcal{B}_n} A(S,\pi) |  T_{S,\pi}| \, \mu(\pi) \biggr]
=
 \biggl[\,  \sum_{i \in [n]} \sum_{\pi \in \mathcal{B}_n}  |  T_{\{ i \},\pi}| \, \mu(\pi) \biggr] 
 \cdot 
 \biggl[ \, \sum_{\pi \in \mathcal{B}_n}     A(S,\pi) \,  \mu(\pi) \biggr].
\end{equation}
\end{lemma}

\begin{proof}
Note first  that by combining Theorem~\ref{theorem: lim sol I}, Remark~\ref{remark: better form of equations system} and Lemma~\ref{lemma: transitive}, it follows that for any \( T \subseteq [n] \), we have
\begin{equation}\label{eq: the system in lemma proof}
\sum_{\pi \in \mathcal{B}_n}     A(T,\pi)
 \,  \mu(\pi)
=
\begin{cases}
 \frac{n \sum_{\sigma \in \{0,1 \}^n} (2\| \sigma \|-n) \chi_T(\sigma)  \,   \nu_{G,\beta,0}(\sigma)}{\sum_{i \in [n]} \sum_{\sigma \in \{0,1 \}^n} (2\| \sigma \|-n) (2\sigma_i-1) \,   \nu_{G,\beta,0}(\sigma)}  & \text{if } |T| \text{ is odd} \cr
 \sum_{\sigma \in \{0,1 \}^n}  \chi_T(\sigma) \, \nu_{G,\beta,0}(\sigma)  & \text{if } |T| \text{ is even.}
\end{cases}
\end{equation}
Next, note that for any such set \( T \subseteq [n] \), we have
\begin{align*}
& \sum_{\sigma \in \{ 0,1 \}^n} (2\| \sigma \|-n) \chi_T(\sigma)\,  \nu_{G,\beta,0}(\sigma) 
= 
\sum_{\sigma \in \{ 0,1 \}^n}  \chi_T(\sigma) \sum_{i \in [n]} (2\sigma_i-1)   \, \nu_{G,\beta,0}(\sigma)
  \\&\qquad = \sum_{i \in [n]}\sum_{\sigma \in \{ 0,1 \}^n}    \chi_{T \triangle \{ i \} }(\sigma) \, \nu_{G,\beta,0}(\sigma)  .
\end{align*}
Since \( |S| \) is odd, \( | S \triangle \{ i \} | \) is even for each  \( i \in [n] \) , and hence, by~\eqref{eq: the system in lemma proof}, we have
\begin{align*}
 &  \sum_{i \in [n]}    \sum_{\sigma \in \{ 0,1 \}^n}   \chi_{S \triangle \{ i \} }(\sigma) \, \nu_{G,\beta,0}(\sigma)  
 =
\sum_{i \in [n] }   \sum_{\pi \in \mathcal{B}_n} A(S \triangle \{ i \}, \pi) \mu(\pi)  
\end{align*}
If \( A(S,\pi) = 1 \) for some \( \pi \in \mathcal{B}_n \), then since \( |S| \) is odd, \( \pi \) has exactly one odd sized partition element, which is equal to the restriction of the partition element \( T_{S, \pi} \) of \( \pi \) to \( S \). Further, since \( |S \triangle \{ j \}| \) is even for each \( j \in [n] \), \( A(S \triangle \{ j \},\pi) = 1 \) if and only if \( j \in T_{S, \pi} \), and hence
\[
\sum_{\pi \in \mathcal{B}_n}   \sum_{j \in [n]}  A(S \triangle \{ j \}, \pi) \, \mu(\pi) =  \sum_{\pi \in \mathcal{B}_n} A(S,\pi) |  T_{S,\pi}| \, \mu(\pi)
\]
Combining the three previous inline equations, we obtain
\begin{equation}\label{eq: one side}
 \sum_{\sigma \in \{0,1\}^n} (2\| \sigma \|-n)  \chi_S(\sigma)   \,  \nu_{G,\beta,0}(\sigma) = \sum_{\pi \in \mathcal{B}_n} A(S,\pi) |  T_{S,\pi}| \, \mu(\pi).
\end{equation}
Plugging this into~\eqref{eq: the system in lemma proof}, and recalling that \( |S| \) is odd by assumption, we get
\begin{align*}
&\sum_{\pi \in \mathcal{B}_n}     A(S,\pi) \,  \mu(\pi) 
=  
\frac{n\sum_{\sigma \in \{0,1\}^n} (2 \| \sigma \|-n) \chi_S(\sigma) \,   \nu_{G,\beta,0}(\sigma)}{\sum_{i \in [n]} \sum_{\sigma \in \{ 0,1 \}^n} (2\| \sigma \|-n) (2\sigma_i-1) \,   \nu_{G,\beta,0}(\sigma)} 
\\&\qquad=
\frac{n\sum_{\pi \in \mathcal{B}_n} A(S,\pi) |  T_{S,\pi}| \, \mu(\pi)}{\sum_{i \in [n]} \sum_{\pi \in \mathcal{B}_n}  |  T_{\{ i \},\pi}| \, \mu(\pi)}.
 \end{align*}
Rearranging this equation, we obtain~\eqref{eq: first version of equality}, which is the desired conclusion.
\end{proof}

The purpose of the next lemma is to provide a simpler way to calculate the sums in~\eqref{eq: first version of equality} when \( \mu   \) is a random cluster model on a finite graph \( G \) by translating such sums into corresponding sums for Bernoulli percolation. 

% It is really not simple to calculate the sums above without this translation, as then, even if we multiply the whole expression with the partition constant, we need to think of the whole pattern to get the correct power of two, instead of looking at local patterns.

\begin{lemma}\label{lemma: Bernoulli translation}
Let \( G = (V,E ) \) be a connected and vertex transitive graph. %Let \(  \mu_{G,r} \) be the random cluster model on \( G \) with parameter \( r \),
%Let \( X^{\beta,h} \) be the Ising model on \( G \), 
%and let  \( Y \) be Bernoulli percolation on \( E \) with parameter 
Fix some \( r \in (0,1) \) 
and define \(  \hat r =   \frac{r}{2-r}\). Further, let \( m \) be the length of the shortest cycle in \( G \). Then there is a constant \( Z''_{G,\hat r}\) such that for any function \( f \colon \mathcal{B}_n \to \mathbb{R} \), we have that
\begin{align*}
& \frac{Z'_{G, r} }{Z''_{G,\hat r}}\sum_{\pi \in \mathcal{B}_V}  f(\pi) \,  \mu_{G,r}(\pi)
 \\&\qquad = 
  \sum_{w \in \{0,1 \}^E} \hat{r}^{\| w \|} (1-\hat r)^{|E| - \| w \|}   \, f(\pi[w]) \, (1 + \mathbb{1}((V,E_w) \text{ contains a cycle})) 
  \\&\qquad\qquad  + O(\hat rx^{m+1}).
 \end{align*}
\end{lemma}

\begin{proof}
 First recall the definition of the random cluster model corresponding to a parameter \( q \geq 0 \); 
\begin{equation*}
	\mu_{G,r,q}(\pi)=\frac{1}{Z'_{G,r,q}} \sum_{w \in \{ 0,1\}^E \colon\atop \pi[w]=\pi}r^{\| w \|}(1-r)^{|E|-\|w \|}  q^{\|\pi[w]\|}, \quad \pi \in \mathcal{B}_V.
\end{equation*}
where   \( Z'_{G,r,q}  \) is a normalizing constant ensuring that this is a probability measure. 
This implies in particular that for any \( \pi \in \mathcal{B}_V \), we have
 \begin{align*}
 &Z'_{G, r,q}  \,  \mu_{G,r,q}(\pi) 
 =
 \sum_{w \in \{ 0,1\}^E \colon\atop \pi[w]=\pi} r^{\| w \|}(1-r)^{|E|-\|w \|}  q^{\|\pi[w]\|}
 \\&\qquad =
 \sum_{w \in \{ 0,1\}^E \colon\atop \pi[w]=\pi} r^{\| w \|}(1-r)^{|E|-\|w \|}  q^{|V| - \| w \|} 
 \\[-4ex]&\qquad\qquad\qquad\qquad\quad \cdot (1 + (q-1)\mathbb{1}((V,E_w) \text{ contains a cycle}))+ O(r^{m+1})
 \\[2ex]&\qquad =
 (1-r)^{|E|}q^{|V|}\sum_{w \in \{ 0,1\}^E \colon\atop \pi[w]=\pi} \,  \biggl( \frac{r}{q(1-r)} \biggr)^{\| w \|}    
  \\[-2ex]&\qquad\qquad\qquad\qquad\qquad\qquad\qquad\quad 
  \cdot (1 + (q-1)\mathbb{1}\bigl((V,E_w) \text{ contains a cycle})\bigr)+ O(r^{m+1}).
 \end{align*}
 If we define \( \hat r = \frac{r}{r + q(1-r)} \), then we obtain
  \begin{align*}
 &Z'_{G, r,q}  \,  \mu_{G,r,q}(\pi)  
 \\&\qquad =
 (1-r)^{|E|}q^{|V|}\sum_{w \in \{ 0,1\}^E \colon\atop \pi[w]=\pi} \,  \biggl( \frac{\hat r}{1 - \hat r} \biggr)^{\| w \|}    
  \\[-2ex]&\qquad\qquad\qquad\qquad\qquad\qquad\qquad\quad 
  \cdot (1 + (q-1)\mathbb{1}((V,E_w) \text{ contains a cycle}))+ O(r^{m+1}). 
 \\&\qquad =
 \frac{(1-r)^{|E|}q^{|V|}}{(1-\hat r)^{|E|}}\sum_{w \in \{ 0,1\}^E \colon\atop \pi[w]=\pi} \,  \hat r^{\| w \|}(1-\hat r)^{|E| - \| w \|}    
  \\[-3ex]&\qquad\qquad\qquad\qquad\qquad\qquad\qquad\quad \cdot (1 + (q-1)\mathbb{1}((V,E_w) \text{ contains a cycle}))+ O(r^{m+1}).
 \end{align*}
In particular, this implies that for ant function \( f \colon \mathcal{B}_n \to \mathbb{R} \), we have
  \begin{align*}
  &Z_{G, r,q}' \sum_{\pi \in \mathbb{B}_V}  f(\pi) \,  \mu_{G,r,q}(\pi) 
 \\&\qquad =
 \frac{(1-r)^{|E|}q^{|V|}}{(1-\hat r)^{|E|}}\sum_{w \in \{ 0,1\}^E  } f(\pi[w])\,  \hat r^{\| w \|}(1-\hat r)^{|E| - \| w \|}  
 \\[-2ex]&\qquad\qquad\qquad\qquad\qquad\qquad\qquad\quad   \cdot (1 + (q-1)\mathbb{1}((V,E_w) \text{ contains a cycle}))+ O(r^{m+1}).
 \end{align*}
  If we let \( q = 2 \) the desired conclusion now follows.  
\end{proof}
 
We now use the previous lemma, Lemma~\ref{lemma: Bernoulli translation}, to give a version of Lemma~\ref{lemma: necessary condition} for Bernoulli percolation. To this end, for \(  r \in (0,1) \) and \( q \geq 0 \),  let \( \hat \mu_{G, r,q}\) be the measure on \( \{ 0,1 \}^E \) defined by
\begin{equation*}
    \hat \mu_{G,\hat r,0}(w)  \coloneqq
    \hat r^{\| w \|} (1 - \hat r)^{|E|- \| w \|}  q^{\| \pi[w] \|}
    , \quad w \in \{ 0,1 \}^E.
\end{equation*}{}
We then have the following lemma.

\begin{lemma}\label{lemma: Bernoulli version}
Let \( n \in  \mathbb{N} \) and let \( G = (V,E) \) be a connected graph on  \( n \) vertices. Let \( m \) be the length of the shortest cycle in \( G \). Further, let \( S \subseteq V \) satisfy \( |S| = 3 \) and let \( \Delta_S^{(m)} \) be the number of length \( m \) cycles in \( G \) which contain all the vertices in \( S \). 
Assume that \( \beta > 0 \) is such that \(  \mu_{G,1-e^{-2\beta}}  \) is a subsequential limit of color representations of \( \nu_{G, \beta,h} \)  as \( h \to 0 \).  Set  \(  \hat {r} =   \frac{r}{2-r}\). Then 
 \begin{equation}\label{eq: Bernoulli version}
 \begin{split}	
& n  
\Bigl[ \, \sum_{w \in \{ 0,1 \}^E} A(S,\pi[w])  \bigl( |  T_{S,\pi[w]}|-1 \bigr)  \,  \hat \mu_{G,\hat r,0}(w)\Bigr]
 +
n \Bigl[ \Delta_S^{(m)} \,  \hat r^m \cdot (m-1) \Bigr]
\\& \qquad =
 \Bigl[ \, \sum_{i \in [n]} \sum_{w \in \{ 0,1 \}^E}  \bigl( |  T_{\{ i \},\pi[w]}|-1 \bigr) \,  \hat \mu_{G,\hat r,0}(w)\Bigr] 
 \cdot 
 \Bigl[ \sum_{w \in \{ 0,1 \}^E}     A(S,\pi[w]) \,   \hat \mu_{G,\hat r,0}(w)\Bigr]
 \\&\qquad\qquad + O(\hat {r}^{m+1}).
\end{split}
\end{equation}
\end{lemma}

\begin{proof}
Note first that by Lemma~\ref{lemma: necessary condition}, we have that
\begin{align*}
&n \cdot \Bigl[ \sum_{\pi \in \mathcal{B}_n} A(S,\pi) |  T_{S,\pi}| \,   \mu_{G,r}(\pi) \Bigr]
\\&\qquad =
 \Bigl[ \sum_{i \in [n]} \sum_{\pi \in \mathcal{B}_n}  |  T_{\{ i \},\pi}| \, \mu_{G,r}(\pi) \Bigr] 
 \cdot 
 \Bigl[ \sum_{\pi \in \mathcal{B}_n}     A(S,\pi) \,   \mu_{G,r}(\pi) \Bigr]
\end{align*}
or equivalently,
\begin{align*}
&\Bigl[ \sum_{w \in \{0,1 \}^n}  n \cdot  \hat \mu_{G,r,2}(w) \Bigr]
\cdot 
\Bigl[ \sum_{w \in \{ 0 ,1 \}^n} A(S,\pi[w])  \bigl( |  T_{S,\pi[w]}|-1 \bigr) \, \hat \mu_{G,r,2}(w) \Bigr]
\\&\qquad =
 \Bigl[ \sum_{i \in [n]} \sum_{w \in \{0,1 \}^n}  \bigl( |  T_{\{ i \},\pi[w]}|-1 \bigr) \, \hat \mu_{G,r,2}(w) \Bigr] 
 \cdot 
 \Bigl[ \sum_{w \in \{0,1 \}^n}     A(S,\pi[w]) \, \hat  \mu_{G,r,2}(w) \Bigr].
\end{align*}

By Lemma~\ref{lemma: Bernoulli translation}, this implies that 
 \begin{align*}
&\Bigl[ \sum_{w \in \{ 0,1 \}^E}  n (1 + \mathbb{1}(\pi[w] \text{ contains a cycle})) \cdot  \hat \mu_{G,\hat r,0}(w) \Bigr]
\\&\qquad \qquad \cdot 
\Bigl[ \sum_{w \in \{ 0,1 \}^E}  A(S,\pi[w])  \bigl( |  T_{S,\pi[w]}|-1 \bigr) (1 + \mathbb{1}(\pi[w] \text{ contains a cycle}))  \,  \hat \mu_{G,\hat r,0}(w)\Bigr]
\\& \qquad =
 \Bigl[ \sum_{i \in [n]} \sum_{w \in \{ 0,1 \}^E}  \bigl( |  T_{\{ i \},\pi[w]}|-1 \bigr) (1 + \mathbb{1}(\pi[w] \text{ contains a cycle})) \,  \hat \mu_{G, \hat r,0}(w)\Bigr] 
 \\&\qquad\qquad\cdot 
 \Bigl[ \sum_{w \in \{ 0,1 \}^E}     A(S,\pi[w])(1 + \mathbb{1}(\pi[w]  \text{ contains a cycle}))  \,   \hat \mu_{G,\hat r,0}(w)\Bigr] + O(\hat r^{m+1}).
\end{align*}
Noting that, by assumption, one needs at least \( m \) edges to make a cycle in \( G \), it follows that the previous equation is equivalent to that
 \begin{align*}
& n  
\Bigl[ \sum_{w \in \{ 0,1 \}^E}  A(S,\pi[w])  \bigl( |  T_{S,\pi[w]}|-1 \bigr)  \,  \hat \mu_{G,\hat r,0}(w)\Bigr]
 +
n  \Bigl[ \Delta_S^{(m)} \,  \hat {r }^m \cdot (m-1) \Bigr]
\\& \qquad =
 \Bigl[ \sum_{i \in [n]} \sum_{w \in \{ 0,1 \}^E}  \bigl( |  T_{\{ i \},\pi[w]}|-1 \bigr) \,  \hat \mu_{G,\hat r,0}(w)\Bigr] 
 \cdot 
 \Bigl[ \sum_{w \in \{ 0,1 \}^E}     A(S,\pi[w]) \,   \hat \mu_{G,\hat r,0}(w)\Bigr]
 \\&\qquad\qquad + O(\hat {r}^{m+1})
\end{align*}
which is the desired conclusion.
\end{proof}

We are now ready to give a proof of Theorem~\ref{theorem: rcm limit}. 

\begin{proof}[Proof of Theorem~\ref{theorem: rcm limit}]
Fix any \( S \subseteq V \) with \( |S| = 3 \).
Note that with the notation of Lemma~\ref{lemma: Bernoulli version}, we have that \( m = 3 \) and that \( \Delta_S^{(3)} = 1 \). 
Moreover, one easily verifies that
\begin{align*}
\sum_{w \in \{ 0,1 \}^E}    A(S,\pi[w]) \,   \hat \mu_{G,\hat r,0}(w) = 
(1 - (1 - \hat r)^3) + (1 - \hat r)^3 \cdot 3(n-3) \hat r^2 + O(\hat r^{3}).
\end{align*}
and that
\begin{align*}
 \sum_{i \in [n]} \sum_{w \in \{ 0,1 \}^E}  \bigl( |  T_{\{ i \},\pi[w]}|-1 \bigr) \,  \hat \mu_{G,\hat r,0}(w)
&= (2-1) \cdot (n-1) \, \hat r ( 1 - \hat{r})^{2(n-2)}
\\&\qquad + (3-1) \cdot \binom{n-1}{2} \binom{3}{1} \hat r^2 
+ O(\hat r^{3}).
\end{align*}

Moreover, using Table~\ref{table: patterns o},  one sees that
\begin{equation}\label{eq: eq with table}
	\begin{split}
\sum_{w \in \{ 0,1 \}^E} A(S,\pi[w])  \bigl( |  T_{S,\pi[w]}|-1 \bigr)  \,  \hat \mu_{G,\hat r,0}(w)
=
3(n-1) \hat r^2 + 
2(7-12 n + 3n^2)   \hat r^3 + O(\hat r^4)
	\end{split}
\end{equation}

\begin{table}[H]
\centering
\hspace{-2.6em} \begin{tabular}{ p{1em} ccc }\cmidrule[0.3mm]{2-4}
 & Pattern &  \( |  T_{S,\pi[w]}| \) & Probability up to \(  O(\hat r^{4}) \)\\ \cmidrule{2-4}
\footnotesize (i) & 
 \adjustbox{valign=c}{\begin{tikzpicture}[scale=0.6]
 
\draw[dotted]  (0,0) -- (1,0) -- (0.5,0.8) -- (0,0);
\draw[] (1,0) -- (1.5,0.8) ;
\draw[] (0,0) -- (-0.5, 0.8) -- (0.5,0.8);

\fill[white] (0,0) circle (3mm);\fill[] (0,0) circle (0.9mm); 
\fill[white] (1,0) circle (3mm);\fill[] (1,0) circle (0.9mm); 
\fill[white] (0.5,0.8) circle (3mm);\fill[] (0.5,0.8) circle (0.9mm); 
\fill[white] (1.5,0.8) circle (0mm);\fill[] (1.5,0.8) circle (0.5mm); 
\fill[white] (-0.5,0.8) circle (0mm);\fill[] (-0.5,0.8) circle (0.5mm); 
\end{tikzpicture}}
 & \( 2 \) &  \( 3(n-3)(n-4) \hat{r}^3\) \\ 
\footnotesize (ii) & 
 \adjustbox{valign=c}{\begin{tikzpicture}[scale=0.6]
 
\draw[dotted]  (0,0) -- (1,0) -- (0.5,0.8) -- (0,0);
\draw[] (0,0) -- (0.5,0.35) -- (1,0);
\draw[] (0.5,0.35) -- (0.5,0.8);

\fill[white] (0,0) circle (3mm);\fill[] (0,0) circle (0.9mm); 
\fill[white] (1,0) circle (3mm);\fill[] (1,0) circle (0.9mm); 
\fill[white] (0.5,0.8) circle (3mm);\fill[] (0.5,0.8) circle (0.9mm); 
\fill[] (0.5,0.35) circle (0.5mm); 
\end{tikzpicture}}
 & \( 4 \) &  \( (n-3) \hat r^3\) \\ 
\footnotesize (iii) & 
 \adjustbox{valign=c}{\begin{tikzpicture}[scale=0.6]
 
\draw[dotted]  (0,0) -- (1,0) -- (0.5,0.8) -- (0,0);
\draw[]  (0,0) -- (1,0); 

\draw[] (0.5,0.8) -- (1.5,0.8);

\fill[white] (0,0) circle (3mm);\fill[] (0,0) circle (0.9mm); 
\fill[white] (1,0) circle (3mm);\fill[] (1,0) circle (0.9mm); 
\fill[white] (0.5,0.8) circle (3mm);\fill[] (0.5,0.8) circle (0.9mm); 
\fill[white] (1.5,0.8) circle (0mm);\fill[] (1.5,0.8) circle (0.5mm); 
\end{tikzpicture}}
 & \( 2 \) &  \( 3 (n-3) \hat{r}^2  (1-\hat{r})^{2(n-2)}\) \\ 
\footnotesize (iv) & 
 \adjustbox{valign=c}{\begin{tikzpicture}[scale=0.6]
 
\draw[dotted]  (0,0) -- (1,0) -- (0.5,0.8) -- (0,0);
\draw[]  (0,0) -- (1,0); 
\draw[] (0.5,0.8) -- (1.5,0.8) -- (2,0);

\fill[white] (0,0) circle (3mm);\fill[] (0,0) circle (0.9mm); 
\fill[white] (1,0) circle (3mm);\fill[] (1,0) circle (0.9mm); 
\fill[white] (2,0) circle (0mm);\fill[] (2,0) circle (0.5mm); 
\fill[white] (0.5,0.8) circle (3mm);\fill[] (0.5,0.8) circle (0.9mm); 
\fill[white] (1.5,0.8) circle (0mm);\fill[] (1.5,0.8) circle (0.5mm); 
\end{tikzpicture}}
 & \( 3 \) &  \( 3(n-3)(n-4) \,  \hat{r}^3\) \\ 
\footnotesize (v) & 
 \adjustbox{valign=c}{\begin{tikzpicture}[scale=0.6]
 
\draw[dotted]  (0,0) -- (1,0) -- (0.5,0.8) -- (0,0);
\draw[]  (0,0) -- (1,0); 
\draw[] (-0.5,0.8) -- (0.5,0.8) -- (1.5,0.8) ;

\fill[white] (-0.5,0.8) circle (0mm);\fill[] (-0.5,0.8) circle (0.5mm); 
\fill[white] (0,0) circle (3mm);\fill[] (0,0) circle (0.9mm); 
\fill[white] (1,0) circle (3mm);\fill[] (1,0) circle (0.9mm); 
\fill[white] (0.5,0.8) circle (3mm);\fill[] (0.5,0.8) circle (0.9mm); 
\fill[white] (1.5,0.8) circle (0mm);\fill[] (1.5,0.8) circle (0.5mm); 
\end{tikzpicture}}
 & \( 3 \) &  \( 3\binom{n-3}{2} \hat{r}^3 \) \\ 
\footnotesize (vi) & 
 \adjustbox{valign=c}{\begin{tikzpicture}[scale=0.6]
 
\draw[dotted]  (0,0) -- (1,0) -- (0.5,0.8) -- (0,0);
\draw[]  (0,0) -- (1,0); 
\draw[] (0.5,0.8) -- (1.5,0.8) -- (1,0);

\fill[white] (0,0) circle (3mm);\fill[] (0,0) circle (0.9mm); 
\fill[white] (1,0) circle (3mm);\fill[] (1,0) circle (0.9mm); 
\fill[white] (0.5,0.8) circle (3mm);\fill[] (0.5,0.8) circle (0.9mm); 
\fill[white] (1.5,0.8) circle (0mm);\fill[] (1.5,0.8) circle (0.5mm); 
\end{tikzpicture}}
 & \( 4 \) &  \( 3\cdot 2(n-3) \hat r^3\) \\ 
\footnotesize (vii) & 
 \adjustbox{valign=c}{\begin{tikzpicture}[scale=0.6]
 
\draw[dotted]  (0,0) -- (1,0) -- (0.5,0.8) -- (0,0);
\draw[] (0.5,0.8) -- (0,0) -- (1,0);

\fill[white] (0,0) circle (3mm);\fill[] (0,0) circle (0.9mm); 
\fill[white] (1,0) circle (3mm);\fill[] (1,0) circle (0.9mm); 
\fill[white] (0.5,0.8) circle (3mm);\fill[] (0.5,0.8) circle (0.9mm); 
\end{tikzpicture}}
 & \( 3 \) &  \( 3 \hat r^2(1 - \hat {r})^{1 + 3(n-3)} \) \\ 
\footnotesize (viii) & 
 \adjustbox{valign=c}{\begin{tikzpicture}[scale=0.6]
 
\draw[dotted]  (0,0) -- (1,0) -- (0.5,0.8) -- (0,0);
\draw[] (0,0) -- (1,0) -- (0.5,0.8) -- (1.5,0.8);

\fill[white] (0,0) circle (3mm);\fill[] (0,0) circle (0.9mm); 
\fill[white] (1,0) circle (3mm);\fill[] (1,0) circle (0.9mm); 
\fill[white] (0.5,0.8) circle (3mm);\fill[] (0.5,0.8) circle (0.9mm); 
\fill[white] (1.5,0.8) circle (0mm);\fill[] (1.5,0.8) circle (0.5mm); 
\end{tikzpicture}}
 & \( 4 \) &  \( 3 \cdot 2 (n-3) \hat r^3 \) \\ 
\footnotesize (ix) & 
 \adjustbox{valign=c}{\begin{tikzpicture}[scale=0.6]
 
\draw[dotted]  (0,0) -- (1,0) -- (0.5,0.8) -- (0,0);
\draw[] (0.5,0.8) -- (1.5,0.8);
\draw[] (0,0) -- (0.5,0.8) -- (1,0);

\fill[white] (0,0) circle (3mm);\fill[] (0,0) circle (0.9mm); 
\fill[white] (1,0) circle (3mm);\fill[] (1,0) circle (0.9mm); 
\fill[white] (0.5,0.8) circle (3mm);\fill[] (0.5,0.8) circle (0.9mm); 
\fill[white] (1.5,0.8) circle (0mm);\fill[] (1.5,0.8) circle (0.5mm); 
\end{tikzpicture}}
 & \( 4 \) &  \( 3(n-3) \hat r^3 \) \\ 
\footnotesize (x) & 
 \adjustbox{valign=c}{\begin{tikzpicture}[scale=0.6]
 
\draw[]  (0,0) -- (1,0) -- (0.5,0.8) -- (0,0);

\fill[white] (0,0) circle (3mm);\fill[] (0,0) circle (0.9mm); 
\fill[white] (1,0) circle (3mm);\fill[] (1,0) circle (0.9mm); 
\fill[white] (0.5,0.8) circle (3mm);\fill[] (0.5,0.8) circle (0.9mm); 
\end{tikzpicture}}
 & \( 3 \) &  \( \hat r^3 \) \\ 
\end{tabular}
\caption{The table above shows the local edge patterns we need to consider (up to permutations) to obtain the expression in~\eqref{eq: eq with table}.}\label{table: patterns o}
\end{table}

This implies in particular that
 \begin{align*}
&\Bigl[ \sum_{w \in \{ 0,1 \}^E} A(S,\pi[w])  \bigl( |  T_{S,\pi[w]}|-1 \bigr)  \,  \hat \mu_{G,\hat r,0}(w)\Bigr]
 +
\Bigl[ \Delta_S^{(m)} \,  \hat {r}^m \cdot (m-1) \Bigr]
\\& \qquad =
3(n-1) \hat r^2 + 
2(8-12 n + 3n^2) \hat r^3 + O(\hat r^4)
\end{align*}
and that
 \begin{align*}
 &
 \Bigl[ \sum_{w \in \{ 0,1 \}^E}  \bigl( |  T_{\{ 1 \},\pi[w]}|-1 \bigr) \,  \hat \mu_{G,\hat r,0}(w)\Bigr] 
 \cdot 
 \Bigl[ \sum_{w \in \{ 0,1 \}^E}     A(S,\pi[w]) \,   \hat \mu_{G,\hat r,0}(w)\Bigr] + O(\hat r^{m+1})
 \\&\qquad = 
 3(n-1) \hat r^2 + 
  6(n-1)(n-3)   \hat r^3 +  O(\hat r^{4}) 
\end{align*}
and hence~\eqref{eq: Bernoulli version} does not hold. By Lemma~\ref{lemma: Bernoulli version}, this implies the desired conclusion.

\end{proof}

%
%
%
%\begin{remark}
%	It is a relatively simple exercise to show that~\eqref{eq: Bernoulli version} in fact holds for all unrooted infinite regular trees (letting \( m = \infty) \). Using this fact, one can subtract the "tree-part" from~\eqref{eq: Bernoulli version} to get a slightly simpler expression, which can be used to reach the same conclusion as above for any finite vertex transitive graph.
%\end{remark}
%

\section{Color representations for large \texorpdfstring{\( h\)}{h}}\label{section: large h}

In this section we show that for any finite graph \( G \) and any  \( \beta > 0 \),  \( \nu_{G,\beta,h} \) has a color representation for all sufficiently large \( h \). The main idea of the proof is to show that a very specific formal solution to~\eqref{eq: linear system}, which was first used in~\cite{fs2019}, is in fact non-negative, and hence a color representation of \( \nu_{G,\beta,h} \), whenever \( h\) is large enough.

\begin{proof}[Proof of Theorem~\ref{theorem: very large h}]
Fix \( n \geq 1 \) and \( \beta > 0 \).

For \( T \subseteq [n] \) with \( |T| \geq 2 \), recall that we let  \( \pi[T] \) denote the unique partition \( \pi \in \mathcal{B}_n \) which is such that \( T \) is a partition element of \( \pi \), and all other partition elements has size one. Let \( \pi[\emptyset] \) denote the unique partition \( \pi \in \mathcal{B}_n \) in which all partition elements are singletons.
Let \( p_h  \coloneqq \nu_{K_n,\beta,h}(1^{\{ 1 \}}) \).
Define
\begin{equation}\label{eq: the formal solution}
\mu(\pi) \coloneqq
\begin{cases}
 \sum_{S \subseteq [n] \colon T \subseteq S}\frac{ (-1)^{|S|-|T|}  \sum_{S' \colon S' \subseteq S} (-(1-p_h))^{|S|-|S'|} \nu_{K_n,\beta,h}(0^{S'})}{p_h(-(1-p_h))^{|S|} + p_h^{|S|}(1-p_h)} &\text{if } \pi = \pi[T],\,  |T| \geq 2 \cr
 1 - \sum_{T \subseteq [n] \colon |T| \geq 2} \mu({\pi[T]}) &\text{if } \pi =  \pi[\emptyset]\cr
0 & \text{else.}
\end{cases}
\end{equation}
By the proof of Theorem~1.6~in~\cite{fs2019}, we have \(   \Phi_p^*(\mu ) = \nu_{K_n,\beta,h} \).
We will show that, given the assumptions, \( \mu(\pi) \geq 0 \) for all \( \pi \in \mathcal{B}_n \), and hence \( \mu \in \Phi_p^{-1}(\nu_{K_n,\beta,h}) \).
To this end, note first that  for \( \sigma \in \{ 0,1 \}^n \), we have
\[ 
	\nu_{K_n,\beta,h}(\sigma ) = 
	Z_{K_n,\beta,h}^{-1} \exp \Bigl(\bigl(\sqrt \beta \sum_{i \in [n]} (2\sigma_i-1) +h/\sqrt \beta \bigl)^2/2 \Bigr),
\]
where \(Z_{K_n,\beta,h}\) is the corresponding normalizing constant. This implies in particular that for any \( S \subseteq [n] \), 
\begin{equation}\label{eq: good asymps}
\begin{split}
		&\lim_{h \to  \infty} \frac{\nu_{K_n,\beta,h}(0^S 1^{[n]\backslash S})}{\nu_{K_n,\beta,h}(1^{[n]})} 
		= \lim_{h \to \infty } e^{-2 |S| (h+\beta  (n-|S|)) }
		= \begin{cases}
			1 &\text{if } S = \emptyset \cr
			0 &\text{else.}
		\end{cases} 
		\end{split}
	\end{equation}
	This implies in particular that  \( \lim_{h \to \infty}p_h =  1 \) and that as \( h \to \infty \), we have
	\[
	Z_{K_n,\beta,h} \sim \nu_{K_n,\beta,h}(0^\emptyset 1^{[n]}) \to 1 .
	\]
	Similarly, one shows that for any \( S \subseteq [n] \), we have
	\begin{equation}
	    \label{eq: two asymptotics}
	\nu_{K_n,\beta,h}(0^S) \sim \begin{cases}
		\nu_{K_n,\beta,h}(0^S1^{[n] \backslash S}) &\text{if } \beta |S| \leq h,  \cr
		\nu_{K_n,\beta,h}(0^{[n]}1^\emptyset)&\text{if } \beta |S| \geq h. \cr
	\end{cases}
	\end{equation}
	Since  \( \beta(n-1)  \leq h\) by assumption, it follows that \( \nu_{K_n,\beta,h}(0^S) \sim \nu_{K_n,\beta,h}(0^S1^{[n] \backslash S})  \) for all \( S \subseteq [n] \), and hence in particular that  \( p_h \sim \nu_{K_n,\beta,h}(0^{[1]}1^{[n]\backslash \{1\}}) \).
	Combining these observations, it follows that  for any \( S \subseteq [n] \) and any \( k \in S \), we have 
	\begin{equation}\label{eq: pre nu1S eq}
	\begin{split}
		&\frac{(1-p_h) \, \nu_{K_n,\beta,h}(0^{S \backslash \{ k \}})}{\nu_{K_n,\beta,h}(0^S)} \sim 
		\frac{\nu_{K_n,\beta,h}(0^{\{1\}}1^{[n]\backslash \{1\}}) \, \nu_{K_n,\beta,h}(0^{S \backslash \{ k \}}1^{[n]\backslash (S \backslash \{ k \})})}{\nu_{K_n,\beta,h}(0^{\emptyset}1^{[n]}) \, \nu_{K_n,\beta,h}(0^S1^{[n]\backslash S})}
		\\&\qquad = e^{-4\beta(|S|-1)}
		\end{split}
	\end{equation}
	and hence when \( h \) is sufficiently large,
	\begin{equation}\label{eq: nu1S asymptotics}
	\begin{split}
		&\nu_{K_n,\beta,h}(0^S) \sim \Bigl[ \prod_{i=1}^{|S|} p e^{4\beta(|S|-i)}  \Bigr] \nu_{K_n,\beta,h}(1^{\emptyset})  
		 = (1-p_h)^{|S|} e^{2 \beta |S| (|S|-1)}.
		\end{split}
	\end{equation}

	We will now use the equations above to describe the behavior of~\eqref{eq: the formal solution}. To this end, fix some \( S \subseteq [n] \) and note  that by combining~\eqref{eq: nu1S asymptotics} and~\eqref{eq: pre nu1S eq}, we obtain
	\begin{align*}
		 &\sum_{S' \colon S' \subseteq S} (-1)^{|S|-|S'|} (1-p_h)^{|S|-|S'|}  \nu_{K_n,\beta,h}(0^{S'}) 
		 \\&\qquad = 
		 (-1)^{|S|}(1-p_h)^{|S|}\sum_{S' \colon S' \subseteq S} (-1)^{|S'|}  e^{-2 \beta |S'| (|S'|-1)} + o(\nu_{K_n,\beta,h}(0^S)).
	\end{align*}
	Using~\eqref{eq: the formal solution}, it follows that if \( T \subseteq [n] \) and \( |T| \geq 2 \), then
\begin{align*}
	 &\mu({\pi[T]})= \sum_{S \subseteq [n]\colon T \subseteq S}\frac{ (-1)^{|S|-|T|}  \sum_{S' \colon S' \subseteq S} (-(1-p_h))^{|S|-|S'|} \nu_{K_n,\beta,h}(0^{S'})}{p_h(-(1-p_h))^{|S|} + p_h^{|S|}(1-p_h)} 
	 \\&\qquad =
	 \sum_{S \subseteq [n] \colon T \subseteq S}\frac{ (-1)^{|S|-|T|} \biggl[ (-(1-p_h))^{|S|}\sum_{S' \colon S' \subseteq S} (-1)^{|S'|}  e^{2 \beta |S'| (|S'|-1)} + o(\nu_{K_n,\beta,h}(0^S))\biggr]}{1-p_h + o(1-p_h)}  
	 \\&\qquad =
	 \frac{1}{1-p_h}\sum_{S \subseteq [n] \colon T \subseteq S}  (-1)^{|T|} (1-p_h)^{|S|}\sum_{S' \colon S' \subseteq S} (-1)^{|S'|}  e^{2 \beta |S'| (|S'|-1)}  + o(\frac{\nu_{K_n,\beta,h}(0^T)}{1-p_h}).
\end{align*}
We  now rewrite the previous equation. To this end, note first that
\begin{align*}
	 &
	 \sum_{S \subseteq [n] \colon T \subseteq S}   (1-p_h)^{|S|}\sum_{S' \colon S' \subseteq S} (-1)^{|S'|}  e^{2 \beta |S'| (|S'|-1)}
	 \\&\qquad =
	  \sum_{S' \subseteq [n]} (-1)^{|S'|}  e^{2 \beta |S'| (|S'|-1)} 
	 \sum_{S \subseteq [n] \colon T \cup S' \subseteq S}   (1-p_h)^{|S|}   
	 \\&\qquad =
	   \sum_{S' \subseteq [n]} (-1)^{|S'|}  e^{2 \beta |S'| (|S'|-1)} 
	    (1-p_h)^{|T \cup S'|} (2-p_h)^{n - |T \cup S'|}  .
\end{align*}
Again using~\eqref{eq: nu1S asymptotics}, it follows that the largest terms in this sum is of order \( \nu_{K_n,\beta,h}(0^T) \), and hence, using symmetry, it follows that the previous equation is equal to
\begin{align*}
	 &
	   \sum_{S' \subseteq [n]\colon |S'| \leq |T|} (-1)^{|S'|}  e^{2 \beta |S'| (|S'|-1)} 
	  \cdot  (1-p_h)^{|T \cup S'|} (2-p_h)^{n - |T \cup S'|}   + o(\nu_{K_n,\beta,h}(0^T)).
\end{align*}
Again using that \( \nu_{K_n,\beta,h} \) is invariant under permutations, it follows that this is equal to
\begin{align*}
	 &
	 \sum_{i=0}^{|T|} (-1)^i  e^{2 \beta i(i-1)} \sum_{j=0}^i \binom{|T|}{j}\binom{n-|T|}{i-j} (1-p_h)^{|T| + (i-j)} (1+p)^{n-(i-j)} + o(\nu_{K_n,\beta,h}(0^T))
	 \\&\qquad =
	 (1-p_h)^{|T|}  \sum_{i=0}^{|T|} (-1)^{i}  e^{2 \beta i(i-1)}   (|T|+n(1-p_h))^i(2-p_h)^{n-i} + o(\nu_{K_n,\beta,h}(0^T)).
\end{align*}
Summing up, we have showed that for any \( T \subseteq [n] \) with \( |T| \geq 2 \), we have that
\begin{equation}\label{eq: another eq}
	 \mu \bigl({\pi[T]}\bigr)
	 = 
	 (1-p_h)^{|T|-1}  \sum_{i=0}^{|T|} (-1)^{|T|-i}  e^{2 \beta i(i-1)}   (|T|+n(1-p_h))^i(2-p_h)^{n-i}   + o( \frac{\nu_{K_n,\beta,h}(0^T)}{1-p_h}).
\end{equation}
Since for any positive and strictly increasing function \( f \colon \mathbb{N} \to \mathbb{R}  \), we have that
\[
 \sum_{i=0}^{|T|} (-1)^{|T|-i}  e^{i f(i)} >    \sum_{i=|T|-1}^{|T|} (-1)^{|T|-i}  e^{i f(i)}   > 0,
\]
%the expression above, not including the error term, is strictly positive. Moreover, it follows that
it follows that
\begin{align*}
 &\sum_{i=0}^{|T|} (-1)^{|T|-i}  e^{2 \beta i(i-1)}   (|T|+n(1-p_h))^i(2-p_h)^{n-i} 
\\&\qquad \geq
\sum_{i=|T|-1}^{|T|} (-1)^{|T|-i}  e^{2 \beta i(i-1)}   (|T|+n(1-p_h))^i(2-p_h)^{n-i} 
\\&\qquad = (|T|+n(1-p_h))^{|T|-1} (2-p_h)^{n-|T|}  e^{2 \beta |T|(|T|-1)} 
\\&\qquad\qquad \cdot \biggl[ \  (|T|+n(1-p_h))
-
e^{-4 \beta (|T|-1)}   (2-p_h) \biggr] .
\end{align*}
This is clearly larger than zero, and in fact, by~\eqref{eq: nu1S asymptotics}, as \( h \) tends to infinity, it is asymptotic to
\[
(1-p_h)^{-|T|} \nu_{K_n,\beta,h}(0^T) \, |T|^{|T|-1}   
\biggl[ \  |T|
-
e^{-4 \beta (|T|-1)}    \biggr].
\]
This implies that the error term in~\eqref{eq: another eq} is much smaller than the rest of the expression, which is strictly positive, i.e. 
\[
\mu\bigl({\pi[T]}\bigr) \sim (1-p_h)^{|T|-1}  \sum_{i=0}^{|T|} (-1)^{|T|-i}  e^{2 \beta i(i-1)}   (|T|+n(1-p_h))^i(2-p_h)^{n-i} >0.
\]
It now remains to show only that \( \mu \bigl(\pi[\emptyset]\bigr) > 0 \). To this end, note first that for any \( T \subseteq [n] \) with \( |T| \geq 2 \), we have that
\begin{align*}
& (1-p_h)^{|T|-1} \sum_{i=0}^{|T|} (-1)^{|T|-i}  e^{2 \beta i(i-1)}   (|T|+n(1-p_h))^i(2-p_h)^{n-i} 
\\&\qquad\leq (1-p_h)^{-1} (2n)^n \sum_{i=0}^{|T|}  (1-p_h)^{|T|-i} (1-p_h)^i e^{2 \beta i(i-1)}  .
\end{align*}
By~\eqref{eq: nu1S asymptotics}, \(  (1-p_h)^i e^{2 \beta i(i-1)} \sim \nu_{K_n,\beta,h}(0^S) \). Since \( \nu_{K_n,\beta,h}(0^S) \) tends to zero as \( h \to \infty \) for any \( S \subseteq [n] \) with \( |S| \geq 1 \), it follows that \( \lim_{h \to \infty} \mu\bigl({\pi[T]}\bigr) = 0 \) provided that \( \lim_{h \to \infty}\nu_{K_n,\beta,h}(0^T)/(1-p_h) = 0 \). To see that this holds, note simply that by~\eqref{eq: nu1S asymptotics},
\begin{align*}
&\nu_{K_n,\beta,h}(0^S)/(1-p_h) 	
\sim
\frac{\nu_{K_n,\beta,h}(0^T1^{[n] \backslash T})}{\nu_{K_n,\beta,h}(0^{[1]} 1^{[n]\backslash [1]})} 
= e^{-2 (|T|-1) (|h|+\beta  (n-|T|-1))}
\end{align*}
which clearly tends to zero as \( h \to \infty \). This concludes the proof.

\end{proof}

\begin{remark}
	The previous proof shows that the formal solution to~\eqref{eq: linear system} given by~\eqref{eq: the formal solution} is a color representation of \( \nu_{K_n,\beta,h} \) whenever \( \beta \) is fixed and \( h \) is sufficiently large. One might hence ask if~\eqref{eq: the formal solution} is also non-negative for fixed \( h \not = 0 \) as \( \beta \to 0 \). This is however not the case, and one can if fact show that for any fixed \( h \not = 0 \), when \( \beta > 0 \) is sufficiently small, \( \mu({\pi[T]}) < 0 \) for any \( T \subseteq [n] \) which is such that \( |T| \) is odd.
\end{remark}

\begin{remark}
	Given that the relationship we get between \( \beta \) and \( h \) in Theorem~\ref{theorem: very large h} is quite far from the conjectured result, one might try to get a stronger result by optimizing the   above proof. However, using Mathematica, one can check that   the particular formal solution to~\eqref{eq: linear system} given by~\eqref{eq: the formal solution} in fact not non-negative when \( (n-1)\beta > h \) when \( n=3,4,5 \).
\end{remark}

%\section{Large \( \beta \)?}

%\textcolor{red}{I cannot apply the not fully supported result here, as it works only for threshold processes.}

\section{Technical lemmas}\label{section: technical lemmas}

In this section we collect and prove the technical lemmas which have been used throughout the paper.

\begin{lemma}\label{lemma: diffs on transitive graphs}
Let \( n \in  \mathbb{N} \) and let \( G = (V,E) \) be a graph with \( n \) vertices. Further, let \( \beta > 0 \) be fixed and,  for \( h \geq 0 \), define \( p_{G, \beta}(h) \coloneqq \nu_{G,\beta,h}(1^{\{ 1 \}}) \). Finally, let \( \sigma \in \{ 0,1 \}^n \).
Then  
\begin{align*}
&\frac{d\nu_{G,\beta,h}(\sigma) }{dp_{G,\beta}} \mid_{h = 0}
= \frac{2 (2\| \sigma \|-n)  \,   \nu_{G,\beta,0}(\sigma)}{\sum_{\hat \sigma \in \{ 0,1 \}^n }  (2 \hat\sigma_1 -1) (2\| \hat \sigma \| -n) \,   \nu_{G,\beta,0}(\hat \sigma)} 
\\&\qquad =
 \frac{2n (2\| \sigma \|-n)  \,   \nu_{G,\beta,0}(\sigma)}{\sum_{\hat \sigma \in \{ -1,1 \}^n } (2\| \hat \sigma \|-n)^2 \,   \nu_{G,\beta,0}(\hat \sigma)}.
\end{align*}
\end{lemma}

\begin{proof}

 By definition, for any \( h \geq 0 \) we have
\[
\nu_{G,\beta,h}(\sigma) = Z_{G,\beta, h}^{-1} \exp(\beta \sum_{\{ i,j\} \in E} (2\sigma_i-1)(2 \sigma_j -1)+ h (2\| \sigma \|-n) ),
\]
where
\[
Z_{G,\beta, h} = \sum_{\hat \sigma \in \{0,1 \}^n} \exp(\beta \sum_{\{i,j\}\in E} (2\hat \sigma_i -1)(2\hat \sigma_j-1)+ h (2\| \hat \sigma \|-n)).
\]
If we differentiate \( Z_{G,\beta, h} \, \nu_{G,\beta,h}(\sigma)  \) with respect to \( h \), we get
\[
\frac{d \bigl(Z_{G,\beta, h} \, \nu_{G,\beta,h}(\sigma)\bigr) }{dh} \mid_{h=0} \; =   (2\| \sigma \|-n)   \,  Z_{G,\beta, 0} \,  \nu_{G,\beta,0}(\sigma).
\]
From this it follows that
\begin{align*}
&\frac{dZ_{G,\beta, h} }{dh} \mid_{h=0} \; =  \sum_{\hat \sigma \in \{ 0,1 \}^n} 
\frac{d \bigl( Z_{G,\beta, h} \, \nu_{G,\beta,h}(\hat \sigma) \bigr) }{dh} \mid_{h=0} 
\\&\qquad =   \sum_{\hat \sigma \in \{ 0,1 \}^n}  \bigl(2\|\hat  \sigma \|-n\bigr)  \,  Z_{G,\beta, 0} \, \nu_{G,\beta,0}(\hat \sigma)
\\&\qquad  
=  Z_{G,\beta, 0} \, \E  \Bigl[ 2\bigl\| X^{G,\beta,0} \bigr\|-n) \Bigr]=  Z_{G,\beta, 0} \cdot 0 = 0.
\end{align*}
As
\[
\frac{d \bigl( Z_{G,\beta, h} \, \nu_{G,\beta,h}(\sigma) \bigr)}{dh} = \frac{dZ_{G,\beta, h} }{dh} \, \nu_{G,\beta,h}(\sigma) + Z_{G,\beta, h}  \,  \frac{d\nu_{G,\beta,h}(\sigma) }{dh}
\]
it follows that
\[
 \frac{d\nu_{G,\beta,h}(\sigma) }{dh} \mid_{h = 0} \; = \bigl(2\| \sigma \|-n\bigr)  \,   \nu_{G,\beta,0}(\sigma)
\]
and hence
\[
 \frac{dp_{G,\beta}}{dh} \mid_{h = 0} 
 \; =\sum_{\hat \sigma \in \{ 0,1 \}^n\colon \atop \hat \sigma_1 = 1} \bigl( 2\| \hat \sigma \|  -n \bigr)\,   \nu_{G,\beta,0}(\hat \sigma).
\]
Combining the two previous equations, we obtain
\[
\frac{d\nu_{G,\beta,h}(\sigma) }{dp_{G,\beta}} \mid_{h=0} = \frac{d\nu_{G,\beta,h}(\sigma) }{dh}  \left(  \frac{dp_{G,\beta}}{dh}\right)^{-1} \mid_{h = 0}
= \frac{\bigl( 2\| \sigma \|-n \bigr)  \,   \nu_{G,\beta,0}(\sigma)}{\sum_{\hat \sigma \in \{ 0,1 \}^n\colon \hat \sigma_1 = 1} \bigl( 2\| \hat \sigma \|-n \bigr)   \,   \nu_{G,\beta,0}(\hat \sigma)}.
\]
The desired conclusion now follows by noting that
\begin{align*}
&\sum_{\hat \sigma \in \{ 0,1 \}^n \colon \hat \sigma_1 = 1} \bigl(2\| \hat \sigma \|-n \bigr)   \,   \nu_{G,\beta,0}(\hat \sigma) 
= \frac{1}{2} \sum_{\hat \sigma \in \{ 0,1 \}^n }  (2\hat\sigma_1-1) \bigl( 2 \| \hat \sigma \|-n \bigr)   \,   \nu_{G,\beta,0}(\hat \sigma) 
\\&\qquad = \frac{1}{2n} \sum_{\hat \sigma \in \{ 0,1 \}^n} \bigl( 2\| \hat \sigma \|-n \bigr)^2   \,   \nu_{G,\beta,0}(\hat \sigma) .
\end{align*}
\end{proof}

\begin{lemma}\label{lemma: sum conversion}
	Let \( n \in \mathbb{N} \). Further, let \( \nu \in \mathcal{P}(\{ 0,1 \}^n) \) and let \( S \subseteq [n] \). Then
	\begin{align*}
	 \sum_{T \colon T \subseteq S} \nu(1^{T}) (-2)^{|T| } 
	 = 
	 \sum_{T \colon T \subseteq S}  (-1)^{|T|} \nu(1^{T}0^{S \backslash T}) .
\end{align*}

\end{lemma}

\begin{proof}
\begin{align*}
	 &\sum_{S' \colon S' \subseteq S} \nu(1^{S'}) (-2)^{|S'| } 
	 = 
	 \sum_{S' \colon S' \subseteq S} \sum_{T \colon T \subseteq S'} \nu(1^{S'}) (-1)^{|S'| } 
	 \\&\qquad = 
	 \sum_{T \colon T \subseteq S}  \sum_{S' \colon T \subseteq S' \subseteq S}\nu(1^{S'}) (-1)^{|S'| } 
	 \\&\qquad = 
	 \sum_{T \colon T \subseteq S}  \sum_{S'' \colon S'' \subseteq  S\backslash T}\nu(1^{T \cup S''}) (-1)^{|T \cup S''| } 
	 \\&\qquad = 
	 \sum_{T \colon T \subseteq S}  (-1)^{|T|} \sum_{S'' \colon S'' \subseteq  S\backslash T}\nu(1^{T \cup S''}) (-1)^{|S''| } 
	 \\&\qquad = 
	 \sum_{T \colon T \subseteq S}  (-1)^{|T|} \nu(1^{T}0^{S \backslash T}) .
\end{align*}
\end{proof}

\begin{lemma}\label{lemma: transitive}
Let \( n \in  \mathbb{N} \) and let \( G \) be a graph with \( n \) vertices. Further,  let \( \beta > 0 \) be fixed and, for \( h \geq 0 \), define \(   p_{G,\beta}(h) \coloneqq \nu_{G,\beta,h}(1^{\{ 1 \}}) \). Finally, let \( S \subseteq [n] \). Then the following two equations hold.
\begin{enumerate}[(i)]
    \item \mbox{} \vspace{-2.5ex}
    \[ \sum_{S' \subseteq S} (-2)^{|S|} \, \nu_{G,\beta,0}(1^{S'}) = (-1)^{|S|}  \sum_{\sigma \in \{ 0,1 \}^n}  \chi_{S}(\sigma)  \, \nu_{G,\beta,0}(\sigma) 
\]
\item \mbox{} \vspace{-5.5ex}
\[ \sum_{S' \subseteq S} (-2)^{|S|-1} \, \frac{d\nu_{G,\beta,h}(1^{S'})}{dp_{G,\beta}}\mid_{h = 0}  = n(-1)^{|S|+1}  \,  \frac{ \sum_{\sigma \in \{0,1 \}^n} \bigl( 2\| \sigma \|-n \bigr) \chi_{S}(\sigma)   \,   \nu_{G,\beta,0}(\sigma)}{\sum_{\sigma \in \{ 0,1 \}^n} \bigl( 2\| \sigma \|-n \bigr)^2 \,   \nu_{G,\beta,0}(\sigma)}.
\]
\end{enumerate}{}

\end{lemma}

\begin{proof}

By Lemma~\ref{lemma: sum conversion}, we have that
\begin{align*}
&\sum_{T \colon T \subseteq S} (-2)^{|T| } \nu_{G,\beta,0}(1^{T}) 
	 = 
	 \sum_{T \colon T \subseteq S}  (-1)^{|T|} \nu_{G,\beta,0}(1^{T}0^{S \backslash T}) 
	 \\&\qquad = 
	 \sum_{T \colon T \subseteq [n]}  (-1)^{|T\cap S|} \,\nu_{G,\beta,0}(1^{T}0^{[n] \backslash T})
	  =
	 \sum_{\sigma \in \{ 0,1 \}^n} \prod_{i \in S} (-(2\sigma_i-1)) \, \nu_{G,\beta,0}(\sigma) \\&\qquad
	 =
	 (-1)^{|S|} \sum_{\sigma \in \{0,1 \}^n} \prod_{i \in S} (2\sigma_i-1) \, \nu_{G,\beta,0}(\sigma) 
	 =
	 (-1)^{|S|} \sum_{\sigma \in \{0,1 \}^n} \chi_S(\sigma) \, \nu_{G,\beta,0}(\sigma) 
\end{align*}
and hence (i) holds.
To see that (ii) holds, note first that by the same argument as above, it follows that
\begin{align*}
&\sum_{T \colon T \subseteq S} (-2)^{|T| -1} \, \frac{d\nu_{G,\beta,h}(1^T)}{dp_{G,\beta}}\mid_{h = 0}
	 = 
	 2^{-1} (-1)^{|S|+1} \sum_{\sigma\in \{ 0,1 \}^n} \chi_S(\sigma) \, \frac{d\nu_{G,\beta,h}(\sigma)}{dp_{G,\beta}}\mid_{h = 0}.
\end{align*}
Applying Lemma~\ref{lemma: diffs on transitive graphs}, we obtain (ii).

\end{proof}

\begin{lemma}\label{lemma: right hand side components}
Let \( n \in  \mathbb{N} \), \( \beta > 0 \) and \( h \geq 0 \). Then
\begin{align*}
&\sum_{\sigma \in \{ 0,1 \}^n} \chi_\emptyset(\sigma) \, \nu_{K_n,\beta,h}(\sigma) = 1
\end{align*}
\begin{align*}
&n \sum_{\sigma \in \{ 0,1 \}^n}   \bigl(2\| \sigma \|-1 \bigr) \, \chi_{\{1\}}(\sigma) \, \nu_{K_n,\beta,h}(\sigma) =  
 \sum_{\sigma \in \{ 0,1 \}^n} \bigl( 2\| \sigma \|-n \bigr)^2 \nu_{K_n,\beta,h}(\sigma)
\end{align*}
\begin{align*}
&(n)_2\sum_{\sigma \in \{ 0,1 \}^n } \chi_{[2]}(\sigma)   \, \nu_{K_n,\beta,h}(\sigma)
=   \sum_{\sigma \in \{ 0,1 \}^n} \bigl( 2\| \sigma\|-n \bigr)^2 \nu_{K_n,\beta,h}(\sigma)  -  n
\end{align*}
\begin{align*}
&(n)_3 \sum_{\sigma \in \{ 0,1 \}^n } \bigl( 2\| \sigma\|-n \bigr)  \, \chi_{[3]}(\sigma)  \, \nu_{K_n,\beta,h}(\sigma)
 \\&\qquad=\sum_{\sigma \in \{ 0,1 \}^n} \bigl( 2\| \sigma\|-n \bigr)^4 \nu_{K_n,\beta,h}(\sigma)   
-
(3n-2)   \sum_{\sigma \in \{ 0,1 \}^n} \bigl( 2\| \sigma\|-n \bigr)^2 \nu_{K_n,\beta,h}(\sigma)  
\end{align*}
\begin{align*} 
&(n)_4 \sum_{\sigma \in \{ 0,1 \}^n} \chi_{[4]}(\sigma) \, \nu_{K_n,\beta,h}(\sigma)  
\\&\qquad =  \sum_{\sigma \in \{ 0,1 \}^n}\bigl( 2\| \sigma\|-n \bigr)^4 \nu_{K_n,\beta,h}(\sigma) 
-
 2(3n-4)\sum_{\sigma \in \{ 0,1 \}^n}  \bigl( 2\| \sigma\|-n \bigr)^2 \nu_{K_n,\beta,h}(\sigma) 
\\&\qquad\qquad+
 3(n^2-2 n)
\end{align*}
and
\begin{align*}
&(n)_5 \sum_{\sigma \in \{ 0,1 \}^n} \bigl( 2\| \sigma\|-n \bigr)  \, \chi_{[5]}(\sigma) \, \nu_{K_n,\beta,h}(\sigma) 
\\&\qquad=
   \sum_{\sigma \in \{ 0,1 \}^n} \bigl( 2\| \sigma\|-n \bigr)^6 \nu_{K_n,\beta,h}(\sigma) 
-
 10( n-2)  \sum_{\sigma \in \{ 0,1 \}^n} \bigl( 2\| \sigma\|-n \bigr)^4 \nu_{K_n,\beta,h}(\sigma) 
\\&\qquad+
(15 n^2 - 50 n+24) \sum_{\sigma \in \{ 0,1 \}^n} \bigl( 2\| \sigma\|-n \bigr)^2 \nu_{K_n,\beta,h}(\sigma) . 
\end{align*}
\end{lemma}

\begin{proof}
By symmetry, for each \( m \in [n] \) we have that
\begin{equation}\label{eq: sum prod equation I}
\begin{split}
&\binom{n}{m} \sum_{\sigma \in \{0,1 \}^n} \chi_{[m]}(\sigma) \, \nu_{K_n,\beta,h}(\sigma) 
=
\sum_{\substack{S \subseteq [n] \colon \\ |S| = m}}
\sum_{\sigma \in \{ 0,1 \}^n} \chi_S(\sigma) \, \nu_{K_n,\beta,h}(\sigma) 
 \\&\qquad =
 \sum_{\sigma \in \{ 0,1 \}^n} \nu_{K_n,\beta,h}(\sigma) \sum_{\substack{S \subseteq [n] \colon \\ |S| = m}} \chi_S(\sigma)
\end{split}
\end{equation}
and, completely analogously, 
\begin{equation}\label{eq: sum prod equation II}
 \binom{n}{m}\sum_{\sigma \in \{ 0,1 \}^n} \bigl( 2\| \sigma \|-n \bigr) \,  \chi_{[m]} (\sigma)\, \nu_{K_n,\beta,h}(\sigma) = \sum_{\sigma \in \{ 0,1 \}^n} \bigl(2 \| \sigma \| -n \bigr) \,  \nu_{K_n,\beta,h}(\sigma)\sum_{\substack{S \subseteq [n] \colon \\ |S| = m}} \chi_S(\sigma).
\end{equation}
Now fix some \( \sigma \in  \{ 0,1 \}^n \). Then  we have that 
\[
\sum_{i \in [n]} \mathbb{1}_{\sigma_i=1} =   \| \sigma \| 
\]
and hence
\begin{align*}
 &  \sum_{S\subseteq [n] \colon  |S| = m} \chi_S(\sigma) 
 = 
\sum_{i=0}^{n-\| \sigma \|} \binom{n-\|  \sigma \|}{i} \binom{ \| \sigma \| }{m-i} (-1)^i.
\end{align*}
When \( m \in [5] \), it follows that
\begin{align*}
& \sum_{S\subseteq [n] \colon |S| = m}   \chi_S(\sigma) 
 = 
 \begin{cases}
\| \sigma\| & \text{if } m = 1\cr
 \frac{\bigl( 2\| \sigma\|-n \bigr)^2}{2!} - \frac{n}{2} &\text{if } m = 2 \cr
 \frac{\bigl( 2\| \sigma\|-n \bigr)^3}{3!} - \frac{(3n-2) \cdot \bigl( 2\| \sigma\|-n \bigr)}{6}  &\text{if } m = 3 \cr
 \frac{\bigl( 2\| \sigma\|-n \bigr)^4}{4!} -\frac{(3n-4) \cdot \bigl( 2\| \sigma\|-n \bigr)^2}{12}+\frac{n^2-2 n}{8}  &\text{if } m = 4 \cr
\frac{\bigl( 2\| \sigma\|-n \bigr)^5}{5!} - \frac{( n-2) \cdot \bigl( 2\| \sigma\|-n \bigr)^3}{12}  + \frac{(15 n^2 - 50 n+24) \cdot \bigl( 2\| \sigma\|-n \bigr)}{120}
  &\text{if } m = 5.
 \end{cases}
\end{align*}
By plugging these expressions into equations~\eqref{eq: sum prod equation I}~and~\eqref{eq: sum prod equation II}, the desired conclusion follows.

\end{proof}

 \section{Acknowledgements}
 
The author acknowledges support from the European Research Council, 
grant no.\ 682537,  from Stiftelsen G S Magnusons fond and from Knut and Alice Wallenbergs stiftelse. 

The author would like to thank Jeffrey E. Steif for many interesting conversations, and in particular for mentioning the example given in Proposition~\ref{theorem: Ising on K3}.
The author would also like to thank the anonymous referee for making several useful comments.

\end{document}